\PassOptionsToPackage{table}{xcolor}
\documentclass[review,onefignum,onetabnum,nohypdvips]{siamart171218}

\usepackage{mathrsfs,amsmath,amssymb,bm,latexsym}
\usepackage{lipsum}
\usepackage{amsfonts}
\usepackage{caption, subcaption, graphicx}
\usepackage{epstopdf}
\usepackage{makecell, rotating, bbding}
\usepackage{multirow}
\usepackage{algorithmic, algorithm}
\usepackage{tikz}
\usetikzlibrary{shapes.geometric, arrows}
\usepackage[draft]{changes}
\ifpdf
  \DeclareGraphicsExtensions{.eps,.pdf,.png,.jpg}
\else
  \DeclareGraphicsExtensions{.eps}
\fi

\usepackage{enumitem}
\setlist[enumerate]{leftmargin=.5in}
\setlist[itemize]{leftmargin=.5in}

\newsiamremark{remark}{Remark}
\newsiamremark{hypothesis}{Hypothesis}
\crefname{hypothesis}{Hypothesis}{Hypotheses}
\newsiamthm{claim}{Claim}
\newsiamremark{defy}{Definition}

\headers{The GADI framework with GRP for large sparse linear systems}{K. Jiang, X. Su, and J. Zhang}

\title{
A general alternating-direction implicit framework with Gaussian process regression parameter prediction for large sparse linear systems
	\thanks{Submitted to the editors DATE.
	\funding{
The work was supported in part by the National Natural Science Foundation of China (12171412, 11771370), Natural Science Foundation for Distinguished Young Scholars of Hunan Province (2021JJ10037), Hunan Youth Science and Technology Innovation Talents Project (2021RC3110), the Key Project of Education Department of Hunan Province (19A500, 21A0116).
	}}}
\author{Kai Jiang\thanks{
	Key Laboratory of Intelligent Computing and Information Processing of Ministry of Education, 
	Hunan Key Laboratory for Computation and Simulation in Science and Engineering, School of Mathematics and Computational Science,
	Xiangtan University, Xiangtan, Hunan, China, 411105.
	(Corresponding authors. \email{kaijiang@xtu.edu.cn} (KJ); \email{zhangjuan@xtu.edu.cn} (JZ)). }
\and Xuehong Su\footnotemark[2]
\and Juan Zhang\footnotemark[2]
	}

\usepackage{amsopn}

\makeatletter
\newcommand*{\addFileDependency}[1]{% argument=file name and extension
  \typeout{(#1)}% latexmk will find this if $recorder=0 (however, in that case, it will ignore #1 if it is a .aux or .pdf file etc and it exists! if it doesn't exist, it will appear in the list of dependents regardless)
  \@addtofilelist{#1}% if you want it to appear in \listfiles, not really necessary and latexmk doesn't use this
  \IfFileExists{#1}{}{\typeout{No file #1.}}% latexmk will find this message if #1 doesn't exist (yet)
}
\makeatother

\usepackage{url}
\usepackage{xcolor}
\definecolor{newcolor}{rgb}{.8,.349,.1}

\DeclareMathOperator*{\argmin}{\mathrm{argmin}}

\definecolor{SlateBlue1}{RGB}{131, 111, 255}
\definecolor{DarkOrange1}{RGB}{255, 127, 0}
\definecolor{DodgerBlue2}{RGB}{28, 134, 238}
\definecolor{DarkCyan}{RGB}{0, 139, 139}
\definecolor{DarkRed}{RGB}{139, 0, 0}
\definecolor{grey11}{RGB}{28, 28, 28}
\definecolor{Red1}{RGB}{255, 0, 0}
\definecolor{DarkMagenta}{RGB}{139, 0, 139}
\definecolor{lightGreen}{RGB}{144, 238, 144}
\definecolor{mgray}{rgb}{0.9,0.9,0.9}

\ifpdf
\hypersetup{
  pdftitle={The GADI framework with GRP for large sparse linear systems},
  pdfauthor={K. Jiang, X. Su, and J. Zhang}
}
\fi

\begin{document}
	\maketitle
	\begin{abstract}
		This paper proposes an efficient general alternating-direction implicit
		(GADI) framework for solving large sparse linear systems. The convergence
		property of the GADI framework is discussed. Most of existing ADI methods can
		be unified in the developed framework. Meanwhile the GADI
		framework can derive new ADI methods. Moreover, as the algorithm efficiency
		is sensitive to the splitting parameters, we offer a data-driven approach,
		the Gaussian process regression (GPR) method based on the Bayesian inference, to
		predict the GADI framework's relatively optimal parameters. The GPR method
		only requires a small training data set to learn the regression prediction
		mapping, can predict accurate splitting parameters, and has high
		generalization capability. It allows us to efficiently solve linear systems
		with a one-shot computation, and does not require any repeated computations.
		Finally, we use the three-dimensional convection-diffusion equation,
		two-dimensional parabolic equation, and
		continuous Sylvester matrix equation to examine the performance of our proposed methods.
		Numerical results demonstrate that the proposed framework
		is faster tens to thousands of times than the existing ADI methods, such as
		(inexact) Hermitian and skew-Hermitian splitting type methods in which the
		consumption of obtaining relatively optimal splitting parameters is ignored.
		As a result, our proposed methods can solve much larger linear systems which these existing ADI methods have not reached.
		%	\keywords {}
	\end{abstract}
	
	
	
\section{Introduction}
\label{sec:intrd}

Large sparse linear systems have wide applications in scientific and
engineering computation. To the best of our knowledge,  the direct methods, such as
Gaussian elimination, QR decomposition, and LU and Cholesky factorizations, have been
deeply researched since the 1970s; refer to \cite{bai2021matrix} and the references
therein.  In practical problems, as the coefficient matrix is often received by
discrete  or integral operators, it is large scale and sparse. The memory
requirements and the difficulties in developing valid parallel implementations can
limit the scope especially for large problems.  Iterative methods are
popular since they have low memory requirements and are easier to parallelize. In
this work, our interest is introducing a general alternating-direction implicit
(GADI) method to solve large sparse linear systems of the form
\begin{equation}
	Ax=b,
\label{eq:Ax=b}
\end{equation}
where $x,b\in\mathbb{C}^{n}$, $A \in \mathbb{C}^{n\times n}$ is a nonsingular matrix.
%When $x$ is a matrix, \cref{eq:Ax=b} could become many kinds of linear matrix equations,
%such as the continuous Sylvester and Lyapunov matrix equations.

\subsection{Background}

Alternating-direction implicit (ADI) methods are widely applied to scientific
computation, such as linear systems, partial differential equations (PDEs), and
optimization. In the 1950s, Peaceman and Rachford proposed an ADI approach, the Peaceman-Rachford splitting (PRS) method,
for solving second-order elliptic equations\,\cite{peaceman1955numerical}.
Subsequently, Douglas and Rachford developed an
efficient method, now named the Douglas-Rachford splitting (DRS) method,
for solving heat conduction problems\,\cite{douglas1956numerical}. Since,
numerous ADI methods have been presented and applied to solving different
PDEs \cite{ greif1998iterative,marchuk1990splitting}.
Besides the application of solving PDEs, ADI methods have also been used in
optimization. For instance, a useful optimization method, the
alternating-direction method of multipliers, equivalent to the
DRS\,\cite{eckstein1992douglas}, has been
widely applied in many fields, such as image science, machine learning, and low-rank
matrix completion \cite{ lions1979splitting, themelis2020douglas}.

The idea of the ADI method has also been devoted to numerical linear
algebra \cite{varga1999matrix}. Recently, Bai, Golub, and Ng have offered a
Hermitian and skew-Hermitian splitting (HSS) method,
analogous to the classical ADI methods in solving PDEs, such as PRS and DRS methods,
for non-Hermitian positive definite linear systems \cite{bai2003hermitian}.
Concretely, they split the coefficient matrix into the Hermitian and non-Hermitian
parts through a splitting parameter, and proved that the HSS method
is convergent unconditionally to the exact solution of \cref{eq:Ax=b}.
%The HSS method might be
%the first attempt to solve the equation \cref{eq:Ax=b} with the ADI method in
%numerical algebra, where the coefficient matrix $A$ is non-Hermitian.

Many researchers have paid much attention to the HSS-type methods in recent years due to
its elegant mathematical property, including the normal
and skew-Hermitian splitting (NSS) method, the positivedefinite and
skew-Hermitian splitting (PSS) method, and the generalized HSS method
\cite{2007On, benner2009on, benzi2003, ber2005}. Nowadays, these HSS-type methods have been applied
to many problems, such as the saddle point system, the matrix equation, the spatial fractional diffusion equation, and the complex
semidefinite linear system \cite{bai2011hermitian,cao2019improved,wang2020preconditioned}.
A systematical introduction of the HSS-type methods can be found in a recent monograph
\cite{bai2021matrix}.

\subsection{Challenges}

The ADI methods have attracted extensive interest. However, there still
exist several challenges in scheme construction and algorithm efficiency.
Existing ADI methods mainly are concern with the concrete matrix splitting formulation.
However, there is a lack of a general framework to unify these existing ADI methods,
even further offering new ADI schemes. It is the first challenge of the development of
ADI methods.
%Although the ADI iterative methods have attracted extensive interest, they only involve the particular case and variant of the ADI iterative methods in Numerical Algebra, such as the HSS iterative method, which is an equivalent form of the PR splitting iterative method. And NSS and PSS Iterative methods, etc., are different splitting forms proposed corresponding to particular problems and still belong to the research category of the ADI method. Thus, it is necessary to present a general ADI (GADI) framework, including these existing methods.

The second challenge is how to choose optimal splitting parameters of ADI methods.
ADI methods require splitting the matrix into different parts with splitting
parameters. The efficiency of ADI methods is very sensitive to
these splitting parameters \cite{axelsson2004a,bai2003hermitian}; also see
\Cref{fig:e1_alpha}. There
are two main methods for selecting parameters.
 One is traversing parameters or experimental determination
within some intervals to obtain relatively optimal
parameters \cite{bai2011hermitian, ke2014a, ZHENG2014145}. The
advantage of this traversal method is that it can obtain relatively accurate optimal
parameters, but obviously, it consumes a lot of extra time.
Meanwhile, the traversing method is impractical in scientific computation.
In practical calculation, it needs to obtain the solution of a one-shot
efficient computation, rather than chasing the best algorithmic performance
through repeated computations.
Another one is estimating relatively optimal parameters through theoretical analysis
\cite{bai2003hermitian,wang2013positive}.
Such a theoretical analysis can directly offer a formulation or an algorithm for evaluating splitting parameters.
However, the theoretical method is available on a case-by-case basis, and
theoretical estimate error heavily affects the algorithmic efficiency.
Meanwhile, the theoretical methods may be hardly applicable as the scale of linear
system increases.

The third challenge is how to apply ADI methods to efficiently solve
large linear systems.
The performance of ADI methods is sensitive to splitting matrices and parameters.
For example, the existing results demonstrate that, when ignoring the
consumption of obtaining the optimal splitting parameters,
the HSS-type methods can solve sparse linear algebraic system of a million levels
at best in about a hundred seconds \,\cite{agh2016}, and a continuous Sylvester matrix equation at most $256$ order in tens to thousands of
seconds\,\cite{bai2011hermitian,ke2014a,wang2013positive}.
It also should be pointed out that the cost of obtaining these optimal splitting
parameters in these ADI methods may be much more expensive than solving the linear system itself.
Therefore, it is urgent to improve further the efficiency of ADI schemes to address large
linear systems.

\subsection{Contribution}

In this work, we are mainly concerned with the development of ADI methods in solving large
sparse linear systems. Our contributions are summarized as follows:

\begin{itemize}
\item We put forward a GADI framework to address sparse linear systems, which is more flexible for choosing the splitting matrices and splitting parameters. The new proposed framework can put most existing ADI algorithms into a unified framework and offer new ADI approaches. As an attempt, in this work
	we present three new ADI methods to solve linear algebraic equations and linear matrix equations.
	The corresponding convergence properties of the GADI framework and proposed ADI
	methods are also presented.
\item We present a data-driven method to predict relatively optimal splitting
	parameters, the Gaussian process regression (GPR) method based on Bayesian
	inference. Concretely, the GPR method learns a mapping from the dimension of
	linear systems to relatively optimal splitting parameters through a training data
	set produced from small-scale systems.
    The GPR method avoids the expensive consumption of traversing parameters, and provides an
    efficient approach to predicting relatively optimal parameters in practical
    computation. It should be emphasized that the GPR method requires a small
	training data set to learning the regression prediction mapping, and has sufficient accuracy and high generalization capability.
\item
   We improve the performance of solving large linear
	systems by combining the GADI framework and the GPR method.
    In this work, we take a three-dimensional
	(3D) convection-diffusion equation, a two-dimensional (2D) parabolic equation (see Appendix A.4), and a continuous Sylvester matrix equation as examples to
	demonstrate the performance of our proposed methods.
	In the comparison, we ignore the consumption of obtaining relatively optimal splitting parameters
	in the classical ADI methods; however, our presented methods can still accelerate the computation
	within a one-shot computation from tens to thousands of times over these methods.
	Furthermore, we can apply our methods to solve much larger sparse linear systems
	which the classical ADI methods have not arrived.
\end{itemize}

\subsection{Organization and Notations}
The rest of the paper is organized as follows. In \cref{sec:GADI}, we
present the GADI framework to solve linear systems and analyze the convergence properties.
Furthermore, we present three new ADI schemes: the GADI-HS, the practical
GADI-HS, and the GADI-AB approaches for linear algebra and matrix equations.
In \cref{sec:paraSelection}, we offer two methods to select the
parameters of ADI methods. The first one is the data-driven GPR method based on Bayesian
inference. The second one is the theoretical estimation method. In
\cref{sec:rslts}, we illustrate the efficiency of the GADI framework and GPR
parameter prediction method through sufficient numerical experiments. Finally, in
\cref{sec:conclu}, we draw some concluding remarks and prospects.

Throughout the paper, the sets of $n \times n$ complex and real matrices are
denoted by $\mathbb{C}^{n\times n}$ and $\mathbb{R}^{n\times n}$, respectively. If $X\in
\mathbb{C}^{n\times n}$, let $X^T$, $X^{-1}$, $X^{*}$, $\|X\|$, $\|X\|_2$, $\|X\|_F$
denote the transpose, inverse, conjugate transpose, determinant, Euclidean norm, and
Frobenius norm of $X$, respectively. The notations $\lambda(X)=(\lambda_{1}(X),
\ldots,\lambda_{n}(X))$, $\sigma(X)=(\sigma_{1}(X), \ldots,\sigma_{n}(X))$,
$\rho(X)$ denote the eigenvalue set, singular value set and spectral radius of
$X$, respectively. The expression $X > 0$ ($X \geq 0$) means that $X$ is
Hermitian (semi-) positive definite. If $X,~Y\in\mathbb{C}^{n\times n}$, $X > Y$
($X \geq Y$) denotes that $X-Y$ is Hermitian (semi-) positive definite.
$\|\cdot\|$ represents the $L_2$-norm of a vector. For all $x \in  \mathbb{C}^{n}$,
denote $\|x\|_{M} =\|Mx\|$.  The induced matrix norm is $\|
X\|_{M}=\|MXM^{-1}\|$. The symbol $\otimes$ denotes the Kronecker product.
$I$ represents the identity matrix.

\section{Algorithm framework}\label{sec:GADI}
In this section, we first propose the GADI framework to solve linear systems and
corresponding convergence analysis. Then we show that these existing ADI schemes
belong to the GADI framework. Finally, we give three new ADI schemes: the
GADI-HS, the practical GADI-HS algorithms for linear algebraic equations, and the GADI-AB method for
matrix equations.

\subsection{GADI framework for linear systems}

In this section, we propose a GADI framework for solving the linear equation \cref{eq:Ax=b}.
%\begin{equation}
%Ax=b,
%\label{eq:Ax=b}
%\end{equation}
%where $x,b\in\mathbb{C}^{n}$, $A \in \mathbb{C}^{n\times n}$ is a nonsingular matrix.
Let $M, N\in\mathbb{C}^{n\times n}$ be splitting matrices such that
$A=M+N$. Given an initial guess $x^{(0)}$ and $\alpha>0$, $\omega\geq 0$,
the GADI framework is
\begin{equation}\label{eq:GADI for Ax=b}
\left\{\begin{aligned}
&(\alpha I+M)x^{(k+\frac{1}{2})}=(\alpha I-N)x^{(k)}+b,\\
&(\alpha I+N)x^{(k+1)}=(N-(1-\omega)\alpha I)x^{(k)}+(2-\omega)\alpha
x^{(k+\frac{1}{2})}
\end{aligned}\right.
\end{equation}
for $k = 0, 1, \ldots$.

Compared with existing ADI schemes, the GADI framework has more degrees of freedom to construct more ADI methods, including  the
splitting formulation of matrices ($M,N$) and an extra splitting parameter
$\omega$. To demonstrate the generality of GADI, \Cref{tab:GADIframework}
shows the connection between concrete ADI schemes and GADI framework.
%Some special cases of the GADI framework are shown in \Cref{tab:GADIframework},
%which demonstrate the generality of our framework from different splitting matrices
%($M,N$) and extrapolation parameter $\omega$.
\begin{table}[!hbtp]
	\centering
\footnotesize{
	\caption{
	\label{tab:GADIframework}
	Some concrete ADI schemes derived from the GADI framework.}
	\begin{tabular}{|c|c|c|}
		\hline
		$\omega$ &Splitting matrices&    ADI method   \\	
		\hline
		&Hermitian and skew-Hermitian&HSS\,\cite{bai2003hermitian}\\
		0 &Positivedefinite and skew-Hermitian&PSS\,\cite{2007On}\\
		&Normal and skew-Hermitian&NSS\,\cite{2007On}\\
		&Triangular and skew-Hermitian&TSS\,\cite{wang2013positive}\\
		\hline
		\textbf{$[0, 2)$}&Hermitian and skew-Hermitian&GADI-HS \cref{eq:GADI-HS}\\
		\hline
		1 &Hermitian and skew-Hermitian&DRS\,\cref{eq:DRS}\\
		\hline
		\textbf{$[0, 2)$}&---&GADI-AB \cref{eq:GADI-AB for cSe}\\
		\hline
	\end{tabular}
}
\end{table}
From \Cref{tab:GADIframework}, one can find that the GADI framework contains these existing ADI
schemes. More significantly, the GADI framework can derive new ADI methods by choosing the splitting
matrices and splitting parameters, such as the DRS, GADI-HS, and GADI-AB methods as
shown in this work.

Next, we analyze the convergence property of the GADI framework.

\begin{lemma}\label{lem:T(alpha)1}
	 \cite{bai2003hermitian} Let $A\in \mathbb{C}^{n\times n}$ be a positive definite matrix, let $M=\frac{1}{2}(A+A^*) $ and $N=\frac{1}{2}(A-A^*)$ be its Hermitian and skew-Hermitian parts. Then, the matrix
	\begin{equation}
	T(\alpha)=(\alpha I+N)^{-1}(\alpha I-M)(\alpha I+M)^{-1}(\alpha I-N)
	\label{eq:T(alpha)}
	\end{equation}
	satisfies $\rho(T(\alpha))<1$ for any $\alpha>0$.
\end{lemma}

\begin{lemma}\label{lem:T(alpha)2}
	Let $M$ be a positive definite matrix, and let $N$ be a positive (semi-) definite matrix. Then, the  matrix defined by \cref{eq:T(alpha)} satisfies $\rho(T(\alpha))<1$ for any $\alpha>0$.
\end{lemma}
\begin{proof}
	Let $\lambda$ and $\mu$ be the eigenvalues of the matrices $M$ and $N$, respectively. Then the eigenvalues of the matrix $T(\alpha)$ have the following form
	\begin{equation*}
		v=\frac{\alpha-\lambda}{\alpha+\lambda}\cdot \frac{\alpha-\mu}{\alpha+\mu}.
	\end{equation*}
	Note that
	\begin{equation*}
		\left|\frac{\alpha-\mu}{\alpha+\mu}\right|\leq 1
	\end{equation*}
	for $N$ is a positive (semi-) definite matrix and $\alpha>0$, it then follows that
	\begin{equation*}
		\rho(T(\alpha))=\max_{\lambda\in \sigma(M)\atop \mu \in \sigma(N)}\left| \frac{\alpha-\lambda}{\alpha+\lambda}\right|\cdot \left|\frac{\alpha-\mu}{\alpha+\mu}\right|\leq \max_{\lambda\in \sigma(M)}\left| \frac{\alpha-\lambda}{\alpha+\lambda}\right|,
	\end{equation*}
	where $\sigma(M)$ and $\sigma(N)$ denote the spectrum of matrices $M$ and $N$, respectively. Analogously, since $M$ is a positive definite matrix, it is easy to verify that $\rho(T(\alpha))<1$.
\end{proof}

Now we discuss the convergence of the GADI framework \cref{eq:GADI for Ax=b}.
\begin{theorem}\label{th:GADI Conver analysis}
   Let $A\in \mathbb{C}^{n\times n}$ be a positive definite matrix, let $M=\frac{1}{2}(A+A^*) $ and $N=\frac{1}{2}(A-A^*)$ be its Hermitian and skew-Hermitian parts; or let $M$ be a positive definite matrix, and let $N$ be a positive (semi-) definite matrix.
	The GADI framework \cref{eq:GADI for Ax=b} is convergent to
	the unique solution $x^{*}$ of the linear equation
	\cref{eq:Ax=b} for any $\alpha>0$ and $\omega\in[0,2)$. Moreover, the
	spectral radius $\rho(T'(\alpha,\omega))$ satisfies
	\begin{equation}\label{eq:rhoT<1}
	\rho(T'(\alpha,\omega))\leq \frac{1}{2}(\rho(T(\alpha))+1)<1,
	\end{equation}
	where
	\begin{equation}\label{eq:T'}
	T'(\alpha,\omega)=(\alpha I+N)^{-1}(\alpha I+M)^{-1}(\alpha ^{2}
	I+MN-(1-\omega)\alpha A),
	\end{equation}
	and $T(\alpha)$ is defined by \cref{eq:T(alpha)}.
\end{theorem}
\begin{proof}
	For $k=0,1,2,\ldots$, by eliminating $x^{(k+\frac{1}{2})}$ from the
	second equation of \cref{eq:GADI for Ax=b}, we can rewrite the GADI framework as
	\begin{equation*}
	x^{(k+1)}=T'(\alpha,\omega)x^{(k)}+G(\alpha,\omega),
	\end{equation*}
	where the iterative matrix $T'(\alpha,\omega)$ is defined by \cref{eq:T'} and
	$$G(\alpha,\omega)=(2-\omega)\alpha (\alpha I+N)^{-1}(\alpha I+M)^{-1}b.$$	
	From \cref{eq:T'}, we obtain
	\begin{equation*}
	\begin{aligned}
	2T'(\alpha,\omega)
	=&~2(\alpha I+N)^{-1}(\alpha I+M)^{-1}(\alpha ^{2} I+MN-(1-\omega)\alpha (M+N))\\
	=&~(2-\omega)(\alpha I+N)^{-1}(\alpha I+M)^{-1}(\alpha I-M)(\alpha I-N)\\
	&~+ \omega(\alpha I+N)^{-1}(\alpha I+M)^{-1}(\alpha I+M)(\alpha I+N)\\
	=&~(2-\omega)T(\alpha)+\omega I,
	\end{aligned}
	\end{equation*}
	where $(\alpha I+M)^{-1}(\alpha I-M)=(\alpha I-M)(\alpha I+M)^{-1}$. Therefore,
	\begin{equation*}
	T'(\alpha,\omega)=\frac{1}{2}[(2-\omega)T(\alpha)+\omega I].
	\end{equation*}
	Hence, we have the following relationship:
	\begin{equation*}
	\mu_k=\frac{1}{2}[(2-\omega)\lambda_k+\omega],\quad \mu_k\in
	\lambda(T'(\alpha,\omega)),\quad \lambda_k\in\lambda(T(\alpha)),
	\end{equation*}
where $\lambda(T'(\alpha,\omega))$ and $\lambda(T(\alpha))$ denote the
	eigenvalue set of the matrix $T'(\alpha,\omega)$ and $T(\alpha)$,
	respectively. Since $\lambda_k,\mu_k\in \mathbb{C}$, let $\lambda_{k}=a+bi,$ where
	$i=\sqrt{-1}$, then
	\begin{equation}\label{eq:|mu_k|}
	\begin{aligned}
	|\mu_k|
	&=\frac{1}{2}\sqrt{(2-\omega)^2
		a^2+\omega^2+2(2-\omega)a\omega+(2-\omega)^2b^2}\\
	&\leq \frac{1}{2}\sqrt{(2-\omega)^2
		(a^2+b^2)+\omega^2+2(2-\omega)\omega\sqrt{a^2+b^2}}\\
	&=\frac{1}{2}[(2-\omega)|\lambda_{k}|+\omega].
	\end{aligned}
	\end{equation}
	Thus, combining \eqref{eq:|mu_k|} with $\omega\in[0,2)$
	, \Cref{lem:T(alpha)1}, and \Cref{lem:T(alpha)2}, we obtain
	\begin{equation*}
	\rho(T'(\alpha,\omega))\leq \frac{1}{2}[(2-\omega)\rho(T(\alpha))+\omega]<1.
	\end{equation*}
	This implies that the GADI framework is convergent to the unique solution $x^{*}$ of
	\cref{eq:GADI for Ax=b}.
\end{proof}

In the subsequent subsection, we will present three new ADI schemes: the GADI-HS, practical
GADI-HS, and GADI-AB methods derived from the GADI framework.

\subsection{GADI-HS scheme for linear algebraic equations}
\label{subsec:GADI-HS}

Assume that the matrix $A\in \mathbb{C}^{n\times n}$ in \cref{eq:Ax=b} is
a large sparse non-Hermitian and positive definite matrix. The first proposed ADI
scheme in the GADI framework is splitting $A$ into Hermitian (H) and skew-Hermitian
(S) parts. A natural selection is
\begin{equation}
H=\frac{A+A^{*}}{2},\quad S=\frac{A-A^{*}}{2}.
\label{eq:HS}
\end{equation}
The GADI framework \cref{eq:GADI for Ax=b} becomes the GADI-HS scheme
\begin{equation}\label{eq:GADI-HS}
\left\{\begin{aligned}
&(\alpha I+H)x^{(k+\frac{1}{2})}=(\alpha I-S)x^{(k)}+b,\\
&(\alpha I+S)x^{(k+1)}=(S-(1-\omega)\alpha I)x^{(k)}+(2-\omega)\alpha
x^{(k+\frac{1}{2})}.
\end{aligned}\right.
\end{equation}
In the implementation of the GADI-HS method, the two linear systems of \cref{eq:GADI-HS} are both
solved by the direct method.
From the properties of H and S, one can find that the GADI-HS method
satisfies the conditions of \Cref{th:GADI Conver analysis}.
Therefore, the GADI-HS scheme is convergent to the unique
solution $x^{*}$ of the linear \cref{eq:Ax=b}, as
\cref{th:GADI-HS Conver analysis} shows.
\begin{theorem}\label{th:GADI-HS Conver analysis}
The GADI-HS method defined by \cref{eq:GADI-HS} is convergent to the unique
solution $x^{*}$ of the linear equation \cref{eq:Ax=b} for any $\alpha>0$ and
$\omega\in[0,2)$. Moreover, the spectral radius $\rho(T(\alpha,\omega))$ satisfies
	\begin{equation*}
	\rho(T(\alpha,\omega))<1,
	\end{equation*}
	where
	\begin{equation}\label{eq:HSiterMatrix}
	T(\alpha,\omega)=(\alpha I+S)^{-1}(\alpha I+H)^{-1}(\alpha ^{2}
	I+HS-(1-\omega)\alpha A).
	\end{equation}
\end{theorem}

The GADI-HS method is a generalized Hermitian and skew-Hermitian (H-S) splitting scheme. We can obtain different H-S splitting methods through varying $\omega$.
	
\textbf{PRS (HSS) scheme}\quad When $\omega=0$ in GADI-HS method \cref{eq:GADI-HS}, we have the PRS
method immediately \cite{peaceman1955numerical}, equivalent to the HSS
method \cite{bai2003hermitian}.
Given an initial guess $x^{0}$ and $\alpha>0$, for $k=0,1,2,\ldots$, until
$x^{(k)}$ is convergent,
\begin{equation}
\left\{\begin{aligned}
& (\alpha I+H)x^{(k+\frac{1}{2})}=(\alpha I-S)x^{(k)}+b,\\
&(\alpha I+S)x^{(k+1)}=(\alpha I-H)x^{(k+\frac{1}{2})}+b.
\label{eq:HSS}
\end{aligned}\right.
\end{equation}
The iterative matrix of PRS (HSS) is
\begin{equation}\label{eq:M(alpha)}
M(\alpha)=(\alpha I+S)^{-1}(\alpha I-H)(\alpha I+H)^{-1}(\alpha I-S).
\end{equation}

\textbf{DRS scheme}\quad
When $\omega=1$ in GADI-HS method \cref{eq:GADI-HS}, we have the DRS iterative method  \cite{douglas1956numerical}.
Given an initial guess $x^{(0)}$ and $\alpha>0$,~for $k = 0, 1, \ldots$, until ${x^{(k)}}$ is convergent,
\begin{equation}
\left\{\begin{aligned}
&(\alpha I+H)x^{(k+\frac{1}{2})}=(\alpha I-S)x^{(k)}+b,\\
&(\alpha I+S)x^{(k+1)}=Sx^{(k)}+\alpha x^{(k+\frac{1}{2})}.
\label{eq:DRS}
\end{aligned}\right.
\end{equation}
The iterative matrix of DRS is
\begin{equation*}\label{eq:M'(alpha)}
M'(\alpha)=(\alpha I+S)^{-1}(\alpha I+H)^{-1}(\alpha ^{2}
I+HS).
\end{equation*}

\Cref{th:GADI-HS Conver analysis} has offered a preliminary conclusion of convergence. Next we analyze the convergent speed of the GADI-HS scheme through exploring
a deep relationship on spectral radii between $T(\alpha,\omega)$
in \cref{eq:T(alpha)} and $M(\alpha)$ in \cref{eq:M(alpha)}.

\begin{theorem}\label{th:|mu| and |lamba|}
Let $A\in \mathbb{C}^{n\times n}$ be a sparse non-Hermitian and positive
definite matrix, and let $H$ and $S$ defined by \cref{eq:HS} be its Hermitian and
skew-Hermitian parts. Assume that
$\alpha>0$, $\omega \in [0,2)$, $\lambda_{k},\mu_k$ are the eigenvalues of
$M(\alpha)$ (see \cref{eq:M(alpha)}) and $T(\alpha,\omega)$ (see \cref{eq:HSiterMatrix}).
Let $\lambda_{k}=a+bi$, where $i=\sqrt{-1}$.
\\
(i) If $|\lambda_k|^2\leq a$, then $$\rho(M(\alpha))\leq \rho(T(\alpha,\omega))<1.$$
(ii) If $a<|\lambda_k|^2~and~ 0<\omega<\dfrac{4a^2-4a+4b^2}{(1-a)^2+b^2}$, then $$\rho(T(\alpha,\omega))<\rho(M(\alpha))<1.$$
\end{theorem}
\begin{proof}
By using \eqref{eq:|mu_k|}, we have
\begin{equation*}
\begin{aligned}
4|\mu_k|^2-4|\lambda_{k}|^2=&~(2-\omega)^2 a^2+\omega^2+2(2-\omega)a\omega+(2-\omega)^2b^2-
4(a^2+b^2)\\
%				= &\omega^2+2(2-\omega)a\omega +(\omega^2-4\omega)(a^2+b^2)\\
= &~\omega^2+4\omega a-2\omega^2
a+a^2\omega^2+b^2\omega^2-4a^2\omega-4b^2\omega\\
= &~[(1-a)^2+b^2]\omega^2-(4a^2-4a+4b^2)\omega.
\end{aligned}
\end{equation*}

(i) As $\omega$ is a nonnegative constant, when $|\lambda_k|^2\leq a$, i.e., $ 4a^2-4a+4b^2\leq 0$, then
\begin{equation*}
|\lambda_k|\leq |\mu_k|.
\end{equation*}
We obtain
\begin{equation*}
\quad\rho(M(\alpha))\leq \rho(T(\alpha,\omega))<1.
\end{equation*}

(ii) When $a<|\lambda_{k}|^2$ i.e., $4a^2-4a+4b^2>0$ and we have $a^2+b^2<1$, then
\begin{equation*}
\begin{aligned}
&0< \omega< \frac{4a^2-4a+4b^2}{(1-a)^2+b^2}<2.\\
\end{aligned}
\end{equation*}
Thus
$$
[(1-a)^2+b^2]\omega^2-(4a^2-4a+4b^2)\omega<0,
$$
which is equivalent to $[(1-a)^2+b^2]\omega< (4a^2-4a+4b^2),$
i.e., $|\mu_k|<	|\lambda_k|.$
Subsequently, we have
\begin{equation*}
\rho(T(\alpha,\omega))<\rho(M(\alpha))<1.
\end{equation*}
The proof is completed.
\end{proof}

\begin{remark}
\Cref{th:|mu| and |lamba|} proves that the GADI-HS scheme converges faster
than or equally to the HSS method.
In case (i), the GADI-HS method is simplified to the HSS method by taking
$\omega=0$.
In case (ii), the spectral radius of the GADI-HS scheme is smaller than the
HSS method, which means the GADI-HS method has a faster convergent speed.
\end{remark}

\subsection{Practical GADI-HS scheme for linear algebraic equations}
\label{sec:pra_GADI-HS}

\subsubsection{Practical GADI-HS scheme}

The GADI-HS method \cref{eq:GADI-HS} requires solving two linear systems
with coefficient matrices $(\alpha I +H)$ and $(\alpha I +S)$.
Directly solving these linear equations would result in heavy computational
cost, especially for large-scale systems.
An approach to overcome this problem is developing efficient iterative methods
to calculate the subproblems.

Concretely, we employ iterative methods to approximately solve $\tilde x^{(k+\frac{1}{2})}$ by
\begin{equation}\label{eq:pra GADI-HS}
	(\alpha I+H)\tilde x^{(k+\frac{1}{2})}\approx(\alpha
	I-S)\tilde x^{(k)}+b
\end{equation}
and approximately solve $\tilde x^{(k+1)}$ through
\begin{equation}\label{eq:pra GADI-HS2}
(\alpha I+S)(\tilde x^{(k+1)}-\tilde x^{(k)})\approx\alpha (2-\omega)(\tilde x^{(k+\frac{1}{2})}-\tilde x^{(k)}).
\end{equation}
\Cref{alg:pra GADI-HS} summarizes the above practical GADI-HS method. Here, we use the conjugate gradient (CG) and the
CG method on normal coefficient equation (CGNE) as the inner iterative methods in
\Cref{alg:pra GADI-HS}. When $\omega=0$, the practical GADI-HS method becomes
the IHSS (inexact HSS) method\,\cite{ bai2011hermitian, bai2003hermitian}.
\begin{algorithm}
		\footnotesize{
\caption{Practical GADI-HS method}\label{alg:pra GADI-HS}
\begin{algorithmic}[1]
	\REQUIRE $\alpha, \omega, \varepsilon, H, S, \tilde r^{(0)}, \tilde \varepsilon_{k}, \tilde{\eta}_k, K_{max}, k=0, \tilde{x}^{(0)}=0$.
	\WHILE {$ \|\tilde r^{(k)}\|^{2}_2 > \|\tilde r^{(0)}\|^{2}_2\varepsilon $ or $ k<K_{max} $}
	\STATE $\tilde r^{(k)}=b-A\tilde x^{(k)}$;
	\STATE	Solve $(\alpha I +H)\tilde z^{(k)}\approx \tilde r^{(k)}$ such that $\tilde p^{(k)}=\tilde r^{(k)}-(\alpha I +H)\tilde z^{(k)}$ satisfies $\|\tilde p^{(k)}\| \leq \tilde\varepsilon _{k} \|\tilde r^{(k)}\|$;
	\STATE Solve $(\alpha I+S)\tilde y^{(k)}\approx \alpha (2-\omega)\tilde z^{(k)}$
	such that $\tilde q^{(k)}=\alpha (2-\omega)\tilde z^{(k)}-(\alpha I+S)\tilde y^{(k)}$ satisfies $\|\tilde q^{(k)}\|\leq \tilde\eta_{k}\|\alpha(2-\omega) \tilde z^{(k)}\|$;
	\STATE	Compute $\tilde x^{(k+1)}=\tilde x^{(k)}+\tilde y^{(k)}$;
	\STATE	$k=k+1$.
	\ENDWHILE
\end{algorithmic}
}
\end{algorithm}

\subsubsection{Convergence of Practical GADI-HS}

In this section, we discuss the convergent properties of the practical GADI-HS method.
It is shown that \cref{eq:pra GADI-HS} and \cref{eq:pra GADI-HS2} are
equivalent to the symmetric form of two-step splitting iteration, which can lead to the
subsequent conclusion.
\begin{lemma}\label{le:pra GADI-HS}
	For linear system \cref{eq:Ax=b}, assume that $A=M_1-N_1=M_2-N_2$. Let $\{\tilde x^{(k)}\}$ be an iterative sequence defined by
	\begin{equation}\label{eq:pra GADI-HS splitting}
	\begin{aligned}
	&\tilde x^{(k+\frac{1}{2})}=\tilde x^{(k)}+\tilde z^{(k)},~~M_1\tilde
	z^{(k)}=\tilde r^{(k)}+\tilde p^{(k)},\\
	&\tilde x^{(k+1)}=\tilde z^{(k+\frac{1}{2})}+\tilde
	x^{(k+\frac{1}{2})},~~M_2\tilde z^{(k+\frac{1}{2})} =\tilde
	r^{(k+\frac{1}{2})}+\tilde q^{(k+\frac{1}{2})},
	\end{aligned}
	\end{equation}
	such that
	$$\|\tilde p^{(k)}\| \leq \tilde \varepsilon _{k} \|\tilde
	r^{(k)}\|, \|\tilde q^{(k+\frac{1}{2})}\| \leq \tilde \eta _{k} \|\tilde r^{(k+\frac{1}{2})}\|,$$ where $\tilde \varepsilon _{k}>0,~ \tilde \eta _{k}>0$, and $\tilde r^{(k)}=b-A\tilde
	x^{(k)},~~\tilde r^{(k+\frac{1}{2})}=b-A\tilde x^{(k+\frac{1}{2})}$.  Then
	\begin{equation}\label{eq:pra GADI-HS scheme}
	\tilde x^{(k+1)}
	=M_2^{-1}N_2M_1^{-1}N_1\tilde x^{k}+M_2^{-1}(N_2M_1^{-1}+I)b+M_2^{-1}(N_2M_1^{-1}\tilde p^{(k)}+\tilde q^{(k+\frac{1}{2})}).
	\end{equation}
\end{lemma}
\begin{proof}
The proof is in Appendix A.1.
\end{proof}

Based on \Cref{le:pra GADI-HS}, we can prove the convergence of the practical
GADI-HS.

\begin{theorem}\label{th:pra GADI-HS conver}
Let $A\in \mathbb{C}^{n\times n}$ be a positive definite matrix, $H$ and $S$
in \cref{eq:HS} be its Hermitian and skew-Hermitian parts, $\alpha$ be
a positive constant, and $\omega\in [0,2)$.  Assume that $\{\tilde x^{(k)}\}$ is
the iterative sequence generated by \Cref{le:pra GADI-HS}, and $\tilde
x^{*}$ is the exact solution of \cref{eq:Ax=b}. Then
\begin{equation*}
	\|\tilde x^{(k+1)}-\tilde x^{*}\|_{M_2}\leq
	(\rho(T(\alpha,\omega))+\theta_1\tilde \varepsilon_{k}(\gamma+\theta_2
	+\theta_1\gamma\eta_{k})+\theta_1\gamma\eta_{k})\cdot\|\tilde x^{(k)}-\tilde x^{*}\|_{M_2}~,k=0,1,2,
\end{equation*}
where $M_1, N_1, M_2, N_2$ are defined in \Cref{le:pra GADI-HS}, and
$$T(\alpha,\omega)=N_2M_1^{-1}N_1M_2^{-1}, ~\theta_1=\|AM_1^{-1}\|,~\theta_2=\|AM_2^{-1}\|,~\gamma=\|M_2M_1^{-1}\|.$$  Moreover, if
\begin{equation}\label{eq:condition}
\rho(T(\alpha,\omega))+\theta_2 \varepsilon_{max}(\gamma+\theta_1+\theta_2\gamma\eta_{max})+\theta_2\gamma\eta_{max}<1,
\end{equation}
then $\{\tilde x^{(k)}\}$ is convergent to $\tilde x^{*}$, where
\begin{equation*}
	\eta_{max}=\max\limits_{k}\left\lbrace \tilde \eta_{k}
	\right\rbrace,~~\varepsilon_{max}=\max\limits_{k}\left\lbrace \tilde \varepsilon
	_{k} \right\rbrace.
\end{equation*}
\end{theorem}

\begin{proof}
Since $x^{*}$ is the exact solution of \eqref{eq:Ax=b}, then
\begin{equation}\label{eq:two x*}
\begin{aligned}
\tilde x^{*}&=M_1^{-1}N_1\tilde x^*+M_1^{-1}b,\\
\tilde x^{*}&=M_2^{-1}N_2\tilde x^*+M_2^{-1}b=M_2^{-1}N_2M_1^{-1}N_1\tilde x^*+M_2^{-1}(N_2M_1^{-1}+I)b.
\end{aligned}
\end{equation}
From \cref{eq:pra GADI-HS scheme}, we have
\begin{equation}\label{eq:x^k+1/2andx^k+1}
\begin{aligned}
&\tilde x^{(k+\frac{1}{2})}=M_1^{-1}N_1\tilde x^{(k)}+M_1^{-1}b+M_1^{-1}\tilde p^{(k)},\\
&\tilde x^{(k+1)}=M_2^{-1}N_2M_1^{-1}N_1\tilde x^{k}+M_2^{-1}(N_2M_1^{-1}+I)b+M_2^{-1}(N_2M_1^{-1}\tilde p^{(k)}+\tilde q^{(k+\frac{1}{2})}).
\end{aligned}
\end{equation}
Subtracting the two equations of \cref{eq:two x*} from those of
\cref{eq:x^k+1/2andx^k+1} reads
\begin{equation}\label{eq:x^k+1-x*andx^k+1/2-x*}
\begin{aligned}
\|\tilde x^{(k+\frac{1}{2})}-\tilde x^*\|_{M_2}&=\|M_1^{-1}N_1(\tilde x^{(k)}-\tilde x^*)+M_1^{-1}\tilde p^{(k)}\|_{M_2}\\
&\leq\|M_2M_1^{-1}N_1M_2^{-1}\|\|\tilde x^{(k)}-\tilde x^*\|_{M_2}+\|M_2M_1^{-1}\|\|\tilde p^{(k)}\|,\\
\|\tilde x^{(k+1)}-\tilde x^{*}\|_{M_2}&=\|M_2^{-1}N_2M_1^{-1}N_1(\tilde x^{(k)}-\tilde x^*)+M_2^{-1}(N_2M_1^{-1}\tilde p^{(k)}+\tilde q^{(k+\frac{1}{2})})\|_{M_2}\\
&\leq\|N_2M_1^{-1}N_1M_2\|\|\tilde x^{(k)}-\tilde x^*\|_{M_2}+\|N_2M_1^{-1}\|\|\tilde p^{(k)}\|+\|\tilde q^{(k+\frac{1}{2})}\|.
\end{aligned}
\end{equation}
Note that
\begin{equation*}
\|\tilde r^{(k)}\|=\|b-A\tilde x^{(k)}\|=\|A(\tilde x^{*}-\tilde x^{(k)})\|\leq
\|AM_2^{-1}\|\|\tilde x^{(k)}-\tilde x^{(*)}\|_{M_2}
\end{equation*}
and
\begin{equation*}
\begin{aligned}
\|\tilde r^{(k+\frac{1}{2})}\|&=
\|b-A\tilde x^{(k+\frac{1}{2})}\|=\|A(\tilde x^{*}-\tilde x^{(k+\frac{1}{2})})\|\leq
\|AM_2^{-1}\|\|\tilde x^{(k+\frac{1}{2})}-\tilde x^{(*)}\|_{M_2}\\
&\leq \|AM_2^{-1}\|
(\|M_2M_1^{-1}N_1M_2^{-1}\|\|\tilde x^{(k)}-\tilde x^*\|_{M_2}+\|M_2M_1^{-1}\|\|\tilde p^{(k)}\|).
\end{aligned}
\end{equation*}
By the definitions of $\{ \tilde p^{(k)}\}$ and $\{\tilde q^{(k+\frac{1}{2})}\}$
in \Cref{le:pra GADI-HS}, we have
\begin{equation}
\|\tilde p^{(k)}\|\leq
\tilde \varepsilon_{k}\|r^{(k)}\|\leq
\tilde \varepsilon_{k}\|AM_2^{-1}\|\|\tilde x^{(k)}-\tilde x^{(*)}\|_{M_2}
\label{eq:p^k norm}
\end{equation}
and
\begin{equation}
\begin{aligned}
\|\tilde q^{(k+\frac{1}{2})}\|\leq&\tilde \eta_{k}\|\tilde r^{(k+\frac{1}{2})}\|\leq \tilde \eta_{k}\|AM_2^{-1}\|(\|M_2M_1^{-1}N_1M_2^{-1}\|\\
&+\tilde\varepsilon_{k}\|M_2M_1^{-1}\|\|AM_2^{-1}\|)\|\tilde x^{(k)}-\tilde x^*\|_{M_2}.\\
\end{aligned}
\label{eq:q^k+1/2 norm}
\end{equation}
Substituting \cref{eq:p^k norm} and \cref{eq:q^k+1/2 norm} into the second
inequality of \cref{eq:x^k+1-x*andx^k+1/2-x*} yields
\begin{equation*}
\begin{aligned}
&\|\tilde x^{(k+1)}-\tilde x^{*}\|_{M_2}
\leq\|N_2M_1^{-1}N_1M_2^{-1}\|\|\tilde x^{(k)}-\tilde x^*\|_{M_2}+\|N_2M_1^{-1}\|\|\tilde p^{(k)}\|+\|\tilde q^{(k+\frac{1}{2})}\|\\
\leq&\|N_2M_1^{-1}N_1M_2^{-1}\|\|\tilde x^{(k)}-\tilde x^*\|_{M_2}+\tilde \varepsilon_{k}\|N_2M_1^{-1}\|\|AM_2^{-1}\|\|\tilde x^{(k)}-\tilde x^{(*)}\|_{M_2}\\
&+\tilde \eta_{k}\|AM_2^{-1}\|(\|M_2M_1^{-1}N_1M_2^{-1}\|
+\tilde \varepsilon_{k}\|M_2M_1^{-1}\|\|AM_2^{-1}\|)\|\tilde x^{(k)}-\tilde x^*\|_{M_2}\\
\leq & (\|T(\alpha,\omega)\|+\theta_2\tilde\varepsilon_k\|N_2M_1^{-1}\|+\theta_2\tilde\eta_{k}\gamma\|N_1M_2^{-1}\|+\theta_2^2\gamma\tilde\varepsilon_{k}\tilde\eta_{k})\|\tilde x^{(k)}-\tilde x^{*}\|_{M_2}\\
\leq&(\|T(\alpha,\omega)\|+\theta_2\tilde \varepsilon_k\|(M_2-A)M_1^{-1}\|+\theta_2\tilde \eta_{k}\gamma+\theta_2^2\gamma\tilde \varepsilon_{k}\tilde \eta_{k})\|\tilde x^{(k)}-\tilde x^{*}\|_{M_2}\\
\leq&(\|T(\alpha,\omega)\|+\theta_2\tilde \varepsilon_k(\|(M_2M_1^{-1}\|+\|(AM_1^{-1}\|)+\theta_2\tilde \eta_{k}\gamma+\theta_2^2\gamma\tilde \varepsilon_{k}\tilde \eta_{k})\|\tilde x^{(k)}-\tilde x^{*}\|_{M_2}\\
\leq&(\rho(T(\alpha,\omega))+\theta_2\tilde
\varepsilon_{k}(\gamma+\theta_1+\theta_2\gamma\tilde\eta_{k})+\theta_2\gamma\tilde\eta_{k})\|\tilde
x^{(k)}-\tilde x^{*}\|_{M_2}.
\end{aligned}
\end{equation*}
From the splitting formulation in \Cref{le:pra GADI-HS}, it is easy to verify that
$\|N_1M_2^{-1}\|\leq1$. Therefore, the fourth inequality in the above expression is hold.
%\begin{equation*}
%\begin{aligned}
%	\|N_1M_2^{-1}\|&=\|(\alpha I-S)((S-\alpha I)[S-(1-\omega)\alpha I]^{-1}(\alpha I+S))^{-1}\|\\
%	&=\|(\alpha I-S)(\alpha I+S)^{-1}[S-(1-\omega)\alpha I](S-\alpha I)^{-1}\|\\
%	&=\|(\alpha I+S)^{-1}[(1-\omega)\alpha I-S]\|\\
%	&\leq\max \limits_{\sigma _{i}\in
%		\sigma (S)}\left | \frac{\sqrt{(1-\omega)^2\alpha^2-\sigma _{i}^2}}{\sqrt {\alpha ^{2}+\sigma_{i}^{2}}}\right |\leq 1
%\end{aligned}
%\end{equation*}
Obviously, if the condition \cref{eq:condition} is met, then $\{\tilde x^{(k)}\}$ is convergent to $\tilde x^{*}$.
\end{proof}

\begin{remark}
	\Cref{th:pra GADI-HS conver} demonstrates that if the subsystem can be solved exactly, i.e.,
	$\left\lbrace \tilde\varepsilon_{k}\right\rbrace $ and $\left\lbrace
	\tilde\eta_{k}\right\rbrace $ being equal to zero, the practical GADI-HS and the GADI-HS have the same convergent speed.
	Further, the convergent speed
	of the GADI-HS (the practical GADI-HS) method is faster than that of the HSS (the
	IHSS) method when $\rho(T(\alpha,\omega))<\rho (M(\alpha))$ according
	to \Cref{th:|mu| and |lamba|}.
\end{remark}

To guarantee the convergence of the practical GADI-HS method in \Cref{th:pra
GADI-HS conver}, it is sufficient to take
$\tilde\varepsilon_k$ and $\tilde\eta_{k}$ such that the conditions of
\Cref{le:pra GADI-HS} are satisfied. Similar to \cite{bai2003hermitian}, we can offer
a feasible way to choose $\tilde \varepsilon _{k}$ and $\tilde\eta_{k}$ below.

\begin{theorem}\label{th:IADIS convegent speed}
Let the assumptions of \Cref{le:pra GADI-HS} be met. Assume that sequences
$\left\lbrace \tau _{1}(k) \right\rbrace $ and $\left\lbrace \tau_{2}(k)
\right\rbrace$ are both nondecreasing and positive satisfying $$\tau_{1}(k)\leq
1,~~\tau_{2}(k)\leq 1,\quad\lim \limits _{k \to \infty} \sup \tau_{1}(k)=\lim \limits
_{k\to \infty}\sup  \tau_{2}(k)=+\infty.$$ Suppose that $\mu _{1}$ and $\mu _{2} $
are real constants in the interval $(0,1)$ satisfying
\begin{equation*}
\tilde \varepsilon_{k}\leq c_{1}\mu_{1}^{\tau_{1}(k)},~~\tilde\eta_{k}\leq c_2
\mu_{2}^{\tau_{2}(k)},~~~k=0,1,2,
\end{equation*}
where $c_{1}$ and $c_{2}$ are nonnegative constants. Then
\begin{equation*}
\|\tilde x^{(k+1)}-\tilde x^{*}\|_{M_2}\leq \sqrt{\rho(T(\alpha,\omega))+\psi\theta_1\mu ^{\tau
	(k)}})^{2}\|\tilde x^{(k)}-\tilde x^{*}\|_{M_2}~,~~k=0,1,2,
\end{equation*}
where $ \tau(k)=\min \left\lbrace {\tau_{1}(k),\tau_{2}(k)} \right\rbrace$, $\mu
=\max \left\lbrace {\mu _{1},\mu _{2}} \right\rbrace $, $\psi=\max \bigg\{\sqrt
{\gamma c_1c_2}\, ,\dfrac{c_1\gamma+\theta_2c_1+c_2\gamma}{2\sqrt
{\rho(T(\alpha,\omega))}} \bigg\} $ with $\gamma$ and $\theta_2$ being defined in
\Cref{th:pra GADI-HS conver}. In particular, we have
\begin{equation*}
\lim \limits _{k \to \infty} \sup \frac{\|\tilde x^{(k+1)}-\tilde x^{*}\|_{M_2}}{\|\tilde x^{(k)}-\tilde x^{*}\|_{M_2}}=\rho(T(\alpha,\omega)).
\end{equation*}
\end{theorem}
\begin{proof}
From \Cref{th:pra GADI-HS conver}, we obtain
\begin{equation*}
\begin{aligned}
&\|\tilde x^{(k+1)}-\tilde x^{*}\|_{M_2}
\leq(\rho(T(\alpha,\omega))+\theta_1\tilde \varepsilon_{k}(\gamma+\theta_2
+\theta_1\gamma\tilde \eta_{k})+\theta_1\gamma\tilde \eta_{k})\|\tilde x^{(k)}-\tilde x^{*}\|_{M_2}\\
\leq &(\rho(T(\alpha,\omega))+\theta_1c_1\mu_1^{\tau_1(k)}(\gamma+\theta_2
+\theta_1\gamma c_2\mu_2^{\tau_2(k)})+\theta_1\gamma
c_2\mu_2^{\tau_2(k)})\|\tilde x^{(k)}-\tilde x^{*}\|_{M_2}\\
\leq &(\rho(T(\alpha,\omega))+\theta_1c_1\mu^{\tau(k)}(\gamma+\theta_2
+\theta_1\gamma c_2\mu^{\tau(k)})+\theta_1\gamma
c_2\mu^{\tau(k)})\|\tilde x^{(k)}-\tilde x^{*}\|_{M_2}\\
=&(\rho(T(\alpha,\omega))+(c_1\gamma+\theta_2c_1+\gamma
c_2)\theta_1\mu^{\tau(k)}+\theta_1^{2}\gamma
c_1c_2\mu^{2\tau(k)}))\|\tilde x^{(k)}-\tilde x^{*}\|_{M_2}\\
\leq&(\rho(T(\alpha,\omega))+2\psi\sqrt{\rho(T(\alpha,\omega))}\theta_1\mu^{\tau(k)}+\psi^2\theta_1^{2}\mu^{2\tau(k)})\|\tilde x^{(k)}-\tilde x^{*}\|_{M_2}\\
=&(\sqrt{\rho(T(\alpha,\omega))}+\psi\theta_1\mu ^{\tau
	(k)})^{2}\|\tilde x^{(k)}-\tilde x^{*}\|_{M_2}.
\label{eq:x^k+1-x^* norm}	
\end{aligned}
\end{equation*}
It immediately holds that
\begin{equation*}
\lim \limits _{k \to \infty} \sup \frac{\|\tilde x^{(k+1)}-\tilde x^{*}\|_{M_2}}{\|\tilde x^{(k)}-\tilde x^{*}\|_{M_2}}=\rho(T(\alpha,\omega)).
\end{equation*}
\end{proof}

\subsection{GADI-AB scheme for linear matrix equations}\label{sec:Sylvester}

In this section, we apply the GADI framework to solve the matrix equation.
We use a representative example, i.e., the continuous Sylvester equation, which has been widely used in control
theory and numerical PDEs (see \cite{ke2014a,wang2013positive,ZHENG2014145} and the references therein), to demonstrate the implementation.
Concretely, the continuous Sylvester equation can be written as
\begin{equation}\label{eq:cSe}
AX+XB=C,
\end{equation}
where $A\in \mathbb{C}^{m\times m}, B\in \mathbb{C}^{n\times n}$, and $C\in
\mathbb{C}^{m\times n}$ are sparse matrices, and $X\in\mathbb{C}^{m\times n}$ is the
unknown matrix. Assume that $A$ and $B$ are positive semidefinite, at least one of
them is positive definite, and at least one of them is non-Hermitian.
Apparently, under these assumptions, the
continuous Sylvester equation \cref{eq:cSe} has a unique solution.

Applying the GADI framework to \cref{eq:cSe} and replacing splitting matrices $M$ and
$N$ in \cref{eq:GADI for Ax=b} with matrices $A$ and $B$ in \cref{eq:cSe}, we
obtain the GADI-AB method. Given an initial guess $X^{(0)}$ and $\hat\alpha>0$, $\hat\omega\geq 0$,
the GADI-AB framework is
\begin{equation}\label{eq:GADI-AB for cSe}
\left\{\begin{aligned}
&(\hat\alpha I+A)X^{(k+\frac{1}{2})}=X^{(k)}(\hat\alpha I-B)+C,\\
&X^{(k+1)}(\hat\alpha I+B)=X^{(k)}(B-(1-\hat\omega)\hat\alpha I)+(2-\hat\omega)\hat\alpha X^{(k+\frac{1}{2})},
\end{aligned}\right.
\end{equation}
where $k=0,1,\ldots$.
The following theorem gives the convergence result of the GADI-AB method.

\begin{theorem}\label{th:GADI-AB convergent analysis}
Let $A \in \mathbb{C}^{n\times n}$ be positive definite matrix and $B \in
\mathbb{C}^{n\times n}$ be semipositive definite matrix. Then the GADI-AB
method \cref{eq:GADI-AB for cSe} is convergent to the unique solution
$X^{*}$ of the continuous Sylvester equation \cref{eq:cSe} for any $\hat \alpha>0$ and $\hat \omega\in[0,2)$.
\end{theorem}

\begin{proof}
It is evident that the GADI-AB method satisfies the conditions of \Cref{th:GADI
Conver analysis} according to the properties of $A$ and $B$. Therefore, the GADI-AB method
converges to the unique solution $x^{*}$ of \cref{eq:cSe}.
\end{proof}

\section{Parameter selection}
\label{sec:paraSelection}

The effectiveness of the ADI schemes is sensitive to the splitting parameters.
How to efficiently and accurately obtain relatively optimal parameters in splitting methods is
still a challenge.
A common approach is traversing parameters within an interval by amounts of numerical
experiments. The traversing method can obtain relatively accurate optimal parameters;
however, it consumes much repetitive computational amount\,\cite{bai2011hermitian, ke2014a, ZHENG2014145}.
Another approach is using theoretical analysis to estimate the splitting
parameters\,
\cite{bai2011hermitian,bai2003hermitian, penzl2000, wang2013positive}.
However, the theoretical method is available on a case-by-case basis, and
the effectiveness relies heavily on the theoretical estimate.
In this section, we provide a data-driven parameter selection method, the
GPR approach  based on the Bayesian inference, which can efficiently obtain accurate
splitting parameters. It is emphasized that the proposed GPR method can be
available to the GADI framework and other splitting schemes.
As a comparison, we present a theoretical method of selecting splitting parameters
for the GADI-HS method as well. It also should be pointed out that the theoretical method is a case-by-case basis. For the practical GADI-HS and GADI-AB methods, there has been no theory to estimate splitting parameters.

\subsection{Gaussian process regression}\label{sec:GPR}

In this section, the GPR method is proposed to estimate the optimal
parameters of the GADI framework.
GPR is a new regression method, which has developed rapidly during the last two decades. Therefore, the GPR method has wide applications and has become a heated issue in machine
learning \cite{fried2001, Rasmussen2010Gaussian}.
It has lots of advantages such as being easy to implement, flexible to nonparameter
infer, and adaptive to obtaining hyperparameters. The relatively optimal ADI parameters might be dependent on other quantities, such as the eigenvalues of the coefficient matrices. In practical implementation, we only require a user-friendly mapping to obtain the relatively optimal ADI parameters. The scale of linear systems is one of the most readily available quantities. Therefore, we choose to learn the mapping between the input dimension $n$ and the output parameter $\alpha$.

\subsubsection{GPR prediction}

%\added[id=ZQ]{
%Before introducing the
%GPR method, we first give the definition of the Gaussian process (GP).
%}

\begin{definition}
	The Gaussian process (GP) is a collection of random variables which follows the joint Gaussian distribution.
\end{definition}

Assume that we have a training set
$D=\{(n_i, \alpha_i) \,|\, i=1,2,\ldots,d\}:=\{\bm n, \bm \alpha\}$,
where $(n_i, \alpha_i)$ is an input-output pair,
$n_i$ is the dimension of the iterative matrix, and $\alpha_i$ is the splitting
parameter in the GADI framework.
If $\alpha_i$ with respect to $n_i$ obeys the GP, then $f(n)=[f(n_1), f(n_2), \ldots, f(n_d)]$ obeys the
$d$-dimensional Gaussian distribution (GD)
\begin{equation*}\label{eq:joint distribution}
	\begin{bmatrix}
		f(n_1) \\ \vdots \\ f(n_d)
	\end{bmatrix}
	\thicksim N
	\begin{pmatrix}
		\begin{bmatrix}
			\mu(n_1) \\ \vdots \\ \mu(n_d)
		\end{bmatrix},
		\begin{bmatrix}
			k(n_1, n_1) & \cdots & k(n_1, n_d)\\
			\vdots      & \ddots & \vdots     \\
			k(n_d, n_1) & \cdots & k(n_d, n_d)\\
		\end{bmatrix}
	\end{pmatrix}.
\end{equation*}

Evidently, GP is determined by its mean function $\mu(n)$ and
covariance function $k(n,n)$.
This is a natural generalization of the GD whose mean and covariance are a vector and a matrix, respectively.
We can rewrite the above GP as
\begin{equation*}\label{eq:gaussian process}
	f(n)\thicksim GP(\mu(n), K(n,n')),
\end{equation*}
where $n, n'$ are any two random variables in the input set $\bm n$. We usually
set the mean function to be zero.

The task of the GPR method is to learn a mapping relationship between the input set $\bm n$ and
output set $\bm \alpha$, i.e., $f(n): \mathbb{N} \mapsto \mathbb{R}$, and
infer the most possible output value $\alpha_*=f(n_*)$ given the new test point $n_*$.
In an actual linear regression problem, we consider the model as
\begin{equation*}\label{eq:linear regression model}
%	\alpha=f(n)+\varepsilon,
	\alpha=f(n)+\eta,
\end{equation*}
where $\alpha$ is the observed value polluted by additivity noise $\eta$ to prevent
the singularity of generated matrix.
Further, assume that $\eta$ follows a GD with zero mean and variance $\sigma^2$,
i.e., $\eta \thicksim N(0, \sigma^2)$. The desirable range of $\sigma$ is
$[10^{-6}, 10^{-2}]$. In this work we take $\sigma=10^{-4}$ in numerical calculations.

Therefore, the prior distribution of observed value $\bm \alpha$ becomes
\begin{equation*}
	\bm \alpha \thicksim N(\bm \mu_\alpha(\bm n), K(\bm n, \bm n)+\sigma^2I_d).
\end{equation*}
The joint prior distribution of observed value $\bm \alpha$ and prediction
$\bm \alpha_*$ becomes
\begin{equation*}
	\begin{bmatrix}
		\bm \alpha \\ \bm \alpha _*
	\end{bmatrix}
	\thicksim N
	\begin{pmatrix}
		\begin{bmatrix}
			\bm \mu_{\alpha} \\ \bm \mu_{\alpha_*}
		\end{bmatrix},
		\begin{bmatrix}
			K(\bm n, \bm n)+\sigma^2I_d & K(\bm n, \bm n_*)\\
			K(\bm n_*, \bm n) & K(\bm n_*, \bm n_*)
		\end{bmatrix}
	\end{pmatrix},
\end{equation*}
where $I_d$ is a $d$-order identity matrix, and $K(\bm n, \bm n)=(k_{ij})$ is a
symmetric positive definite covariance matrix with $k_{ij}=k(n_i,n_j)$.
$K(\bm n, \bm n_*)$ is a symmetric covariance matrix between
the training set $\bm n$ and test set $\bm n_*$.

Using the Bayesian formula
\begin{equation}\label{eq:bayesian formula}
p(\bm \alpha _*|\bm \alpha)=\frac{p(\bm \alpha|\bm \alpha _*)p(\bm \alpha _*)}{p(\bm \alpha)},
\end{equation}
the joint posterior distribution of prediction is
\begin{equation}\label{eq:GPR}
	\bm \alpha _*|\bm n, \bm \alpha, \bm n_* \thicksim N\left(\bm \mu _*, \bm \sigma^2_*\right),
\end{equation}
where
\begin{equation*}
\begin{aligned}
	\bm \mu_*&=K(\bm n_*, \bm n)\left[K(\bm n, \bm n)+\sigma^2I_d\right]^{-1}(\bm \alpha -\bm \mu_{\alpha})+\bm \mu_{\alpha_*},\\
	\bm \sigma^2_*&=K(\bm n_*, \bm n_*)-K(\bm n_*, \bm n)\left[K(\bm n, \bm n)+\sigma^2I_d\right]^{-1}K(\bm n, \bm n_*).
\end{aligned}
\end{equation*}
The derivation procedure can refer to \cite{Rasmussen2010Gaussian}. For the output $\alpha_*$ in the test
set, one can use the mean value of the above GP as its estimated value,
i.e., $\widehat{\alpha}_*=\mu_*$.

The kernel function is the key of the GPR method, which
generates the covariance matrix to measure the distance between any two input
variables. When the distance is closer, the correlation
of the corresponding output variables is greater. Therefore, we need to choose or construct
the kernel function according to actual requirements. The most commonly used kernel functions include the radial basis
function, the rational quadratic kernel function, the exponential kernel
function, and the periodic kernel function. For more kernel functions refer to \cite{Rasmussen2010Gaussian}.
In this work, we choose the exponential kernel function
\begin{equation}\label{eq:EK}
k(x,y)=\sigma_f^2 \exp\Big(\frac{-\|x-y\|}{2\iota^2}\Big),
\end{equation}
where $\theta=\{\iota,\sigma_f\}$ is the hyperparameter, $ \|x-y\|=\sqrt{\sum\limits_{i}(x_i-y_i)^2}$. The optimal hyperparameter is determined by maximizing the marginal-log likelihood function $p(y|x,\theta)$, i.e.,
\begin{equation}\label{eq:max_log}
	L=\log p(y|x,\theta)=-\frac{1}{2}y^T[K+\sigma^2I_d]^{-1}y-\frac{1}{2}\log ||K+\sigma^2I_d||-\frac{n}{2}\log 2\pi.
\end{equation}	
Maximizing $L$ corresponds to an unconstrained nonlinear optimization problem.
We use the L-BFGS method to address it, which costs little time.

\subsubsection{The implementation of GPR}

Since the GPR results from probability distribution, the regression and prediction
have a specific confidence interval.  The confidence interval is defined as follows,
%\begin{definition}[Confidence interval]
%	A confidence interval estimate is a range of values constructed from the sampling
%	data such that the population parameter is likely to occur within a range with a specified probability. The specified probability is called the level of confidence.
%\end{definition}

\begin{definition}[confidence interval]
Suppose a data set $x_1,\ldots, x_n$ is given, modeled as the realization of random variables $X_1,\ldots,X_n$. Let $\alpha$ be the parameter of interest and $\gamma$ be a number between $0$ and $1$. If there exist sample statistics $L_n = g(X_1,\ldots,X_n)$ and $U_n = h(X_1,\ldots,X_n)$ such that
$$ P(L_n <\theta <U_n) = \gamma $$
for every value of $\alpha$, then $(l_n, u_n)$,
where $l_n = g(x_1,\ldots,x_n)$ and $u_n = h(x_1,\ldots,x_n)$, is called a $100\gamma$
confidence interval for $\alpha$. The number $\gamma$ is called the confidence level.
\end{definition}
The number $l_n$ ($u_n$) is called a $100 \gamma\%$ lower (upper) confidence bound for parameter $\alpha$.
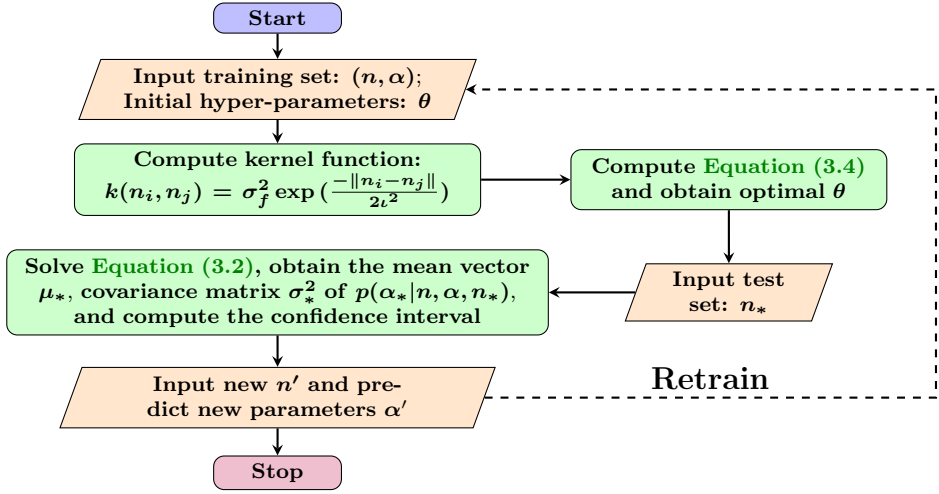
\begin{figure}[!hbpt]
	\centering
	\begin{center}	
		\footnotesize{
		\tikzstyle{A} = [rectangle, rounded  corners,minimum width=1.5cm,minimum height=0.4cm,text width=1.5cm,text centered,draw=black,fill=blue!25]
		\tikzstyle{B} = [trapezium, trapezium left angle=70, trapezium right angle=110, minimum width=4.5cm,minimum height=0.5cm, text width=4.5cm,text centered,draw=black,fill=orange!20]
		\tikzstyle{C} = [rectangle, rounded  corners, minimum width=5cm,minimum height=0.8cm, text width=5.2cm,text centered,draw=black,fill=green!20]
		\tikzstyle{D} = [rectangle, rounded  corners, minimum width=4.0cm,minimum height=0.6cm, text width=4cm,text centered,draw=black,fill=green!20]
		\tikzstyle{E} = [trapezium, trapezium left angle=70, trapezium right angle=110, minimum width=2cm,minimum height=0.6cm, text width=2cm,text centered,draw=black,fill=orange!20]
		\tikzstyle{F} = [rectangle, rounded  corners, minimum width=7cm,minimum height=1cm, text width=7.0cm,text centered,draw=black,fill=green!20]
		\tikzstyle{H} = [trapezium, trapezium left angle=70, trapezium right angle=110, minimum width=5cm,minimum height=0.5cm, text width=5cm,text centered,draw=black,fill=orange!20]
		\tikzstyle{I} = [rectangle, rounded  corners,minimum width=1.5cm,minimum height=0.4cm,text width=1.5cm,text centered,draw=black,fill=purple!25]
		\tikzstyle{arrow} = [thick=50cm,->,>=stealth]		
		\begin{tikzpicture}[node distance=1.7cm]
		minimum	\node (start0) [A] {\textbf{Start}};
		\node (str1) [B,below of=start0,yshift = 0.75cm] {\textbf{Input training set:} \bm{ $(n, \alpha)$}; \textbf{Initial hyper-parameters:} \bm{$\theta$}};
		\node (str2) [C,below of=str1,yshift = 0.5cm, xshift=0cm] {\textbf{Compute kernel function: } \bm{$k(n_i,n_j)=\sigma_f^2 \exp(\frac{-\|n_i-n_j\|}{2\iota^2})$}};
		\node (str3) [D,below of=str1,yshift = 0.5cm, xshift=6cm] {\textbf{Compute  \Cref{eq:max_log} and obtain optimal} \bm{$\theta$}};
		\node (str4) [E,below of=str3,yshift = 0.2cm, xshift=0cm] {\textbf{Input test set:} \bm{$n_*$}};
		\node (str5) [F,below of=str3,yshift=0.2cm,xshift=-6cm] {\textbf{Solve \Cref{eq:GPR}, obtain the mean vector} \bm{$\mu_*$}, \textbf{covariance matrix} \bm{$\sigma^2_*$} \textbf{of} \bm{$p(\alpha_*|n,\alpha ,n_*)$}, \textbf{and compute the confidence interval}};
		\node (str7) [H,below of=str5,yshift = 0.3cm, xshift=-0cm] {\textbf{Input new} \bm{$n'$} \textbf{and predict new parameters} $\bm \alpha'$};
		\node (stop8) [I,below of=str7,yshift = 0.7cm, xshift=0cm] {\textbf{Stop}};
		\draw [arrow,draw=black, thick] (start0) -- (str1);
		\draw [arrow,draw=black, thick] (str1) -- node [right] {}(str2);

		\draw [arrow,draw=black, thick] (str2) -- node[above]{}(str3);
		
		\draw [arrow,draw=black, thick] (str3) -- node[right]{}(str4);
		\draw [arrow,draw=black, thick] (str4) -- node[]{}(str5);
		\draw [arrow,draw=black, thick] (str5) -- node [right] {}(str7);
        \draw [arrow,draw=black,dashed] (str7.east)node%dashed,dotted line
         [right=3cm][above]{\large\color{black}\textbf{Retrain}} --++(6cm, 0cm)|-(str1.east);
        \draw [arrow] (str7) -- node[]{}(stop8);
		\end{tikzpicture}
		}
	\end{center}
	\caption{Flow chart of GPR parameters prediction.}
	\label{fig:GPR}
\end{figure}

\Cref{fig:GPR} summarizes the process of GPR. Here, the confidence interval
denotes the area of the normal distribution of mean $\mu$ and standard deviation
$\sigma$ falling in the interval
$(\mu-1.96\sqrt{\mbox{diag}(\sigma)},\mu+1.96\sqrt{\mbox{diag}(\sigma)})$ accounts
for about $95\%$. Ninety-five percent of the sample means selected from a population will be within 1.96 standard deviations of the population mean $\mu$.
From \Cref{fig:GPR}, it can be seen that the GPR method has established a mapping
between the matrix scale $n$ and the parameter $\alpha$. We input known data to train
an inference model and then predict the unknown parameters.
The known relatively optimal parameters in the training data set come from small-scale linear systems, while the
unknown parameter belongs to that of large linear systems.
To predict the parameter more
accurately and extensively, we also put the predicted data into the training set
to form the retraining set. Subsequently, we use the latest model to predict
the parameters of much larger linear systems.

\subsection{Quasi-optimal parameter of GADI-HS}\label{sec:QOP}

In this section, we offer a theoretical way to select the quasi-optimal parameter of
the GADI-HS method. We first estimate an upper bound of the iterative matrix spectral
radius of the GADI-HS method.
\begin{lemma}\label{le:bound of GADI-HS}
	Let $A\in \mathbb{C}^{n\times n}$ be a positive definite matrix, $H$  and
	$S$ defined by \cref{eq:HS} be its Hermitian and skew-Hermitian parts, $\alpha$ be a positive constant, and $\omega \in [0,2)$.  Then the spectral radius $\rho(T(\alpha,\omega))$ is bounded by
	\begin{equation}\label{eq:bound of T}
	\rho(T(\alpha,\omega))\leq  \frac{\alpha^2+\alpha|1-\omega|\|A\|_2+\lambda_{max}\sigma_{max}}{\alpha(\alpha+\lambda_{min})}=\delta(\alpha^*, \omega^*),
	\end{equation}
	where $T(\alpha,\omega)$ is defined in \cref{eq:HSiterMatrix}, $\sigma _{max}$ is
	the maximum singular value of $S$, and $\lambda_{min}$ and $\lambda _{max}$ are the
	minimum and maximum eigenvalues of $H$, respectively.
\end{lemma}
\begin{proof}
	The proof is in Appendix A.2.
\end{proof}

\begin{lemma}[\cite{bai2003hermitian}]\label{lem:HSSalpha}
	Let $A\in \mathbb{C}^{n\times n}$ be a positive definite matrix,  $H$  and
	$S$ defined by \cref{eq:HS} be its Hermitian and skew-Hermitian parts, and $\alpha$ be a positive constant. Then the quasi-optimal parameter of the HSS iterative method is
	\begin{equation*}
	\alpha^*=\sqrt{\lambda_{min}\lambda_{max}},\quad and \quad \sigma(\alpha^*)=\frac{\sqrt{\lambda_{max}-\sqrt{\lambda_{min}}}}{\sqrt{\lambda_{max}+\sqrt{\lambda_{min}}}}=\frac{\sqrt{\kappa(H)}-1}{\sqrt{\kappa(H)}+1},
	\end{equation*}
	where  $\lambda _{min}$ and $\lambda _{max}$ are the minimum and maximum eigenvalues of $H$, and $\kappa(H)$ is the spectral condition number of $H$.
\end{lemma}

\begin{theorem}[the quasi-optimal parameter of GADI-HS]\label{th:GADI-HSalpha}\quad
	Let $A\in \mathbb{C}^{n\times n}$ be a positive definite matrix, $H$ and
	$S$ defined by \cref{eq:HS} be  its Hermitian and skew-Hermitian parts,
	$\alpha>0$, and $\omega\in [0,2)$, and $\lambda_{k}$ be the eigenvalue of the
	iterative matrix $M(\alpha)$. For the quasi-optimal parameter
	($\alpha^*,\omega^*$) of the GADI-HS method, we have the following conclusions.
	
	(i) When $|\lambda_k|^2\leq a$, then
	\begin{equation*}
	\omega^*=0,\quad \alpha^*=\sqrt{\lambda_{min}\lambda_{max}},\quad and \quad
	\delta(\alpha^*,\omega^*)=\sigma(\alpha^*),
	\end{equation*}
	where $\sigma(\alpha^*)$ is defined in \Cref{lem:HSSalpha}, and $\lambda _{min}$ and $\lambda _{max}$ are the minimum and maximum eigenvalues of $H$, respectively.
	
	(ii) When $a<|\lambda_k|^2~and ~ 0<
	\omega<\dfrac{4a^2-4a+4b^2}{(1-a)^2+b^2}$, then
	\begin{equation*}
	\omega^*=1,\quad \alpha^*=
	\frac{p+\sqrt{p^2+\lambda_{min}^2p}}{\lambda_{min}},
	\end{equation*}
	and
	\begin{equation*}
	\delta(\alpha^*,\omega^*)=\frac{2p^2+2\lambda_{min}^2p+2p\sqrt{p^2+\lambda_{min}^2p}}{2p^2+2\lambda_{min}^2p+2p\sqrt{p^2+\lambda_{min}^2p}+\lambda_{min}^2\sqrt{p^2+\lambda_{min}^2p}}<1,
	\end{equation*}
	where $p=\lambda_{max}\sigma_{max}$,  $\sigma _{max}$ is the maximum singular value of $S$.
\end{theorem}
\begin{proof}
The proof is in Appendix A.3.
\end{proof}
\begin{remark}
	From \Cref{le:bound of GADI-HS}, it implies that the bound of
	$\rho(T(\alpha,\omega))$ is related to the singular value of $S$ and the
	eigenvalue of $H$. However, directly estimating $\alpha$ for the GADI-HS
	method is difficult in actual implementation. Here we minimize the upper bound to
	obtain a quasi-optimal $\alpha$ but not the optimal $\alpha$.
\end{remark}

\section{Numerical experiments}	\label{sec:rslts}

In this section, we present extensive numerical examples to show the power of the GADI
framework and the GPR method. Concretely, we take a 3D convection-diffusion equation, 2D parabolic equation (see Appendix A.4), and continuous Sylvester equation as examples to demonstrate the efficiency of
GADI-HS, practical GADI-HS, and GADI-AB.
All computations are carried out using MATLAB 2018a on a Mac laptop with a 2.3 GHz
CPU Intel Core i5 and 8G memory.
We employ the GPR method to predict the optimal splitting
parameters for the GADI framework.
The whole procedure of GPR prediction
only takes about $1\sim 2$ seconds in the offline training.
From numerical calculations, we find that $\omega$ is insensitive to the order of linear
systems. Thus we can obtain the relatively optimal $\omega$ in small-scale linear systems,
then apply the GPR method to predict the relatively optimal $\alpha$.

\subsection{Three-dimensional convection-diffusion equation}

Consider the following 3D convection-diffusion equation:
\begin{equation}
-(u_{x_1 x_1}+u_{x_2 x_2}+u_{x_3 x_3})+(u_{x_1}+u_{x_2}+u_{x_3})=f(x_1,x_2,x_3)
\label{eq:3DConvDiff}
\end{equation}
on the unit cube $\Omega = [0, 1] \times [0, 1] \times [0, 1]$
with Dirichlet-type boundary condition \cite{bai2003hermitian}.
We use the centered difference method to discretize the convective-diffusion equation
\cref{eq:3DConvDiff}, and obtain the linear system $Ax=b$.  The coefficient matrix is
\begin{equation}
A=T_{1}\otimes I\otimes I+I\otimes T_{2}\otimes I+I\otimes I\otimes T_{3},
\label{eq:NM2}
\end{equation}
where $T_{1}$, $T_{2}$ and $T_{3}$ are tridiagonal matrices.
$T_{1}=\mbox{Tridiag}(t_{2},t_{1},t_{3})$, $T_{2}=T_{3}
=\mbox{Tridiag}(t_{2},0,t_{3})$, $t_{1}=6$, $t_{2}=-1-\beta$, $t_{3}=-1+\beta$,
$\beta=1/(2n+2)$. $n$ is the degree of freedom along each dimension,
	$x\in\mathbb{R}^{n^3}$  is the unknown vector of discretizing $u(x_1,x_2,x_3)$.
	$b\in\mathbb{R}^{n^3}$  is the discretization vector of $f(x_1,x_2,x_3)$ which
is determined by choosing the exact solution $x_e = (1,1,\ldots,1)^{T}$.
All tests are started with the zero vector.
All iterative methods are terminated if the relative residual error satisfies
\begin{align}
	\mbox{RES}={\|r^{(k)}\|_{2}}/{\|r^{(0)}\|_{2}} \leq 10^{-6},
	\label{eq:resPDE}
\end{align}
where $r^{(k)}=b-A x^{(k)}$ is the $k$-step residual.

\subsubsection{A comparison of GADI-HS and HSS}

First, we examine the efficiency of the GADI framework in solving the above
linear algebra equation. Concretely, we apply the HSS and GADI-HS methods to
solve $Ax = b$.
The parameters in the HSS and GADI-HS methods are obtained
from \Cref{lem:HSSalpha} and \Cref{th:GADI-HSalpha}, respectively.
\Cref{tab:alpha_q} shows the numerical results, where ``IT" and ``CPU"
denote the required iterations and the CPU time (in seconds), respectively.
\vspace{0cm}

\begin{table}[!htbp]
	\footnotesize{
	\caption{
	Results of the HSS and the GADI-HS methods for solving 3D convection-diffusion
	equation with theoretical quasi-optimal splitting parameters.\label{tab:alpha_q}}
	}
	\centering
	\footnotesize{
		\begin{tabular}{|c|c|c|c|c|c|c|}
			\hline
			&\multicolumn{3}{c|}{HSS} &
			\multicolumn{3}{c|}{GADI-HS} \\
			\hline
%				\cmidrule{2-4}\cmidrule{5-7}
			{$n^3$} &{$\alpha_q$}&	{IT}  & {CPU(s)} &{$(\alpha_q,\omega_q)$}  &
			{IT}  & {CPU(s)} \\
			\hline
			%			\bm{$8^3$}&2.0521&37&0.0093&4.2645e-07&(0.6208, 1.0)&29& 0.0056&4.5769e-07\\
			%			\bm{$12^3$}&1.4359&52&0.1289&4.2022e-07&(0.4468, 1.0)&39&0.0584&4.0411e-07\\
			%			\bm{$16^3$}&1.1025&66&1.8674&4.4259e-07&(0.3465, 1.0)&48&0.9940&4.1636e-07\\
			%			\bm{$20^3$}&0.8943&79&9.1727&4.9813e-07&(0.2823, 1.0)&56&4.8256&4.8651e-07\\
			%			\bm{$24^3$}&0.7520&92&30.8901&4.5684e-07&(0.2380, 1.0)&65&16.8941&4.3099e-07\\
			{$8^3(512)$}&2.0521&37&0.03&(0.6208,1.0)&29& 0.02\\
			{$12^3(1728)$}&1.4359&52&0.13&(0.4468,1.0)&39&0.06\\
			{$16^3(4096)$}&1.1025&66&1.87&(0.3465,1.0)&48&0.99\\
			{$20^3(8000)$}&0.8943&79&9.17&(0.2823,1.0)&56&4.82\\
			{$24^3(13824)$}&0.7520&92&30.89&(0.2380,1.0)&65&16.69\\
			\hline
		\end{tabular}
	}
\end{table}

From \Cref{tab:alpha_q}, one can find that the GADI-HS method spends about half the
CPU time compared with the HSS method. It is consistent with \Cref{th:GADI-HSalpha}, which shows the convergence speed of the GADI-HS scheme is faster than the HSS method. These results demonstrate that
the GADI-HS scheme derived from the new GADI framework accelerates the
convergence speed in solving \cref{eq:3DConvDiff}.
However, we can find that both methods cost much time for relatively large
linear systems.

\subsubsection{Applying the Practical GADI-HS method}
\label{subsubsec:PGADIHS}

Next, we use the practical GADI-HS (\Cref{alg:pra GADI-HS}) and the IHSS
methods to solve much larger linear systems.
Meanwhile, we employ the GPR method to predict the splitting parameters in the
practical GADI-HS method.
The inner subsystems \cref{eq:pra GADI-HS} and \cref{eq:pra GADI-HS2}
are solved by the CG and the CGNE methods, respectively.
The inner iteration is terminated if residuals satisfy
\begin{equation*}
\|p^{(k)}\|_{2} \leq 1 \times
10^{-\delta_{H}}\|r^{(k)}\|_{2},~~\|q^{(k)}\|_{2}\leq 1 \times
10^{-\delta_{S}}\|r^{(k+\frac{1}{2})}\|_{2},
\end{equation*}
where $\delta _{H}$ and $\delta _{S}$ are controllable tolerances in the inner iteration
to balance the Hermitian and skew-Hermitian parts in the linear subproblems.
In these tests, we set $\delta_{H}=\delta_{S}=2$.

There has been no theoretical approach to estimate splitting parameters for inexact ADI methods, such as the practical GADI-HS and the IHSS schemes.
For the IHSS method, the splitting parameters can be obtained by the traversing method,
be determined experimentally \,\cite{bai2003hermitian,benner2009on,wang2013positive}, or obtained directly by the GPR method as presented in \cref{subsubsec:IHSSGRP}.
Here, we use the traversing method to obtain the relatively optimal parameter
$\alpha^*$ of IHSS in the traversing interval $(0,3]$ with a step size of $0.01$.
We use the GPR method to predict the splitting
parameter $\alpha$ for the practical GADI-HS scheme.
\Cref{tab:PGADIHSsetting} gives the training, test, and
retrained data sets.
$\alpha$ in the training data set of the GPR approach is produced by
traversing parameters as the IHSS method does, but for small-scale linear
systems, $n$ from $2$ to $66$ with different step size $\Delta n$.
The matrix order  $n$ in the test data set is from $1$ to $120$ with $\Delta n=1$, while $n$ in the retrained data set is from $70$ to $120$ with $\Delta n =6$.
\Cref{fig:e1_GRP} shows the optimal parameter regression and prediction processes for
the practical GADI-HS method. From \Cref{fig:e1_GRP}, we can find that
adding predicted points into the training set which forms
the retrained set can shrink the confidence interval.
This improves the prediction accuracy and strengthens the
generalization ability of the regression model.
Since $\omega$ is insensitive to the scale of matrix in the practical GADI-HS, we obtain the optimal $\omega_p=1.9$ by the traversing method for small-scale systems.

\begin{table}[!hbtp]
	\centering
	\footnotesize{
	\caption{The training, test, and retrained data sets of the practical GADI-HS method
		in GPR algorithm.\label{tab:PGADIHSsetting}}
		\begin{tabular}{|c|c|}
			\hline
			& $n: 2\sim10$, $\Delta n=2$\\
		{Training set}  	& $n: 12\sim20$, $\Delta n=4$\\
			& $n: 28\sim44$, $\Delta n=8$\\
			& $n: 56\sim66$, $\Delta n=10$\\
			\hline
			{Test set} & $n: 1\sim120$, $\Delta n =1$     \\			
			\hline
			{Retrained set}&$n: 70\sim120$, $\Delta n=6$ \\
			\hline
		\end{tabular}
	}
\end{table}

\begin{figure}[hbtp]
	\centering
	\includegraphics[width=4.25cm]{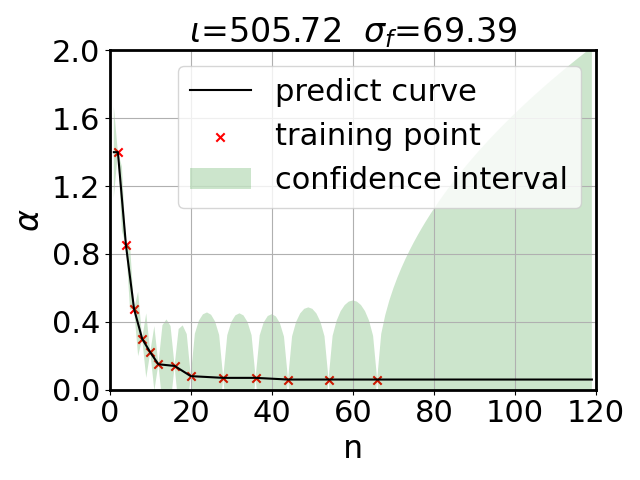}
	\includegraphics[width=4.25cm]{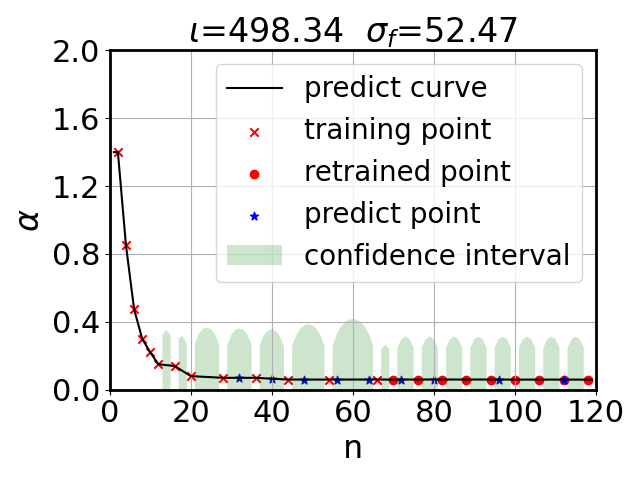}
	\caption{The regression curve of $\alpha$ against $n$ for the practical GADI-HS method.
		\label{fig:e1_GRP}}
\end{figure}

\Cref{tab:e1_e2_result} presents the numerical results obtained by the IHSS and
practical GADI-HS methods for different scale discretization systems.
$IT_{CG}$ and $IT_{CGNE}$ denote the iterations for solving subsystems in
inexact methods. These results show that the IHSS method can solve relatively large linear systems
($n^3$ from $32^3$ to $64^3$) with less
CPU time than the HSS algorithm does when the relatively optimal parameter $\alpha^*$ is used.
However, obtaining $\alpha^*$ in the IHSS costs lots of traversal time, as
shown in the last column in \Cref{tab:e1_e2_result}.
For example, when $n=64^3$,  the traversal time of the IHSS method is more than
$7000$ seconds. As $n$ further increases, the traversal times would become unaffordable.
\begin{table}[!hbpt]
	\centering
	\footnotesize{
	\caption{Results of the IHSS method ($\alpha^*$ obtained by traversing interval
	$(0,3]$ with a step size of $0.01$) and the practical GADI-HS method
	(relatively optimal parameters $(\alpha_p,\omega_p)$ obtained by the GPR algorithm) for
	solving 3D convection-diffusion equation.
		\label{tab:e1_e2_result}}
		\begin{tabular}{|c|lccc|c|}
			\hline
			{Method}&{$n^3$}&	{$\alpha^*$}&{IT}&{CPU}&{Traversal} \\
			&&&(${IT_{CG}}$, ${IT_{CGNE}}$)&{(s)}&{CPU(s)} \\
			\hline
			&{$32^3(32768)$}&0.93&185 (4.19, 1.00)&0.83&477.77\\
			{IHSS}&{$48^3(110592)$}&0.90&369 (4.02, 1.00)&4.44& 2239.66\\
			&{$64^3(261144)$}&0.89&612 (3.95, 1.00)&20.43&7025.53\\
			\hline
			&{$n^3$}&({$\alpha_p,\omega_p$})&{IT}&{CPU}&{Traversal} \\
			&&&(${IT_{CG}, IT_{CGNE}}$)&{(s)}&{CPU(s)} \\
			\hline
			&{$32^3(32768)$}&(0.0699,1.9)&23 (23.35, 2.00)&0.26&0\\
			&{$48^3(110592)$}&(0.0599,1.9)&33 (21.45, 1.21)&1.09&0\\
			{Practical}&{$64^3(261144)$}&(0.0599,1.9)&54 (22.72, 1.07)&4.21&0\\
			{GADI-HS}& {$96^3(884736)$ }&(0.0596,1.9)& 110 (19.75, 1.00)&32.48&0\\
			& {$128^3(2097152)$}&(0.0595,1.9)&186 (13.92, 1.00)&113.82&0\\
			& {$216^3(10077696)$}&(0.0595,1.9)&478 (12.64, 1.00)&1539.48&0\\
			\hline
		\end{tabular}
	}
\end{table}
\begin{figure}[htbp]
	\centering
	\includegraphics[width=13cm]{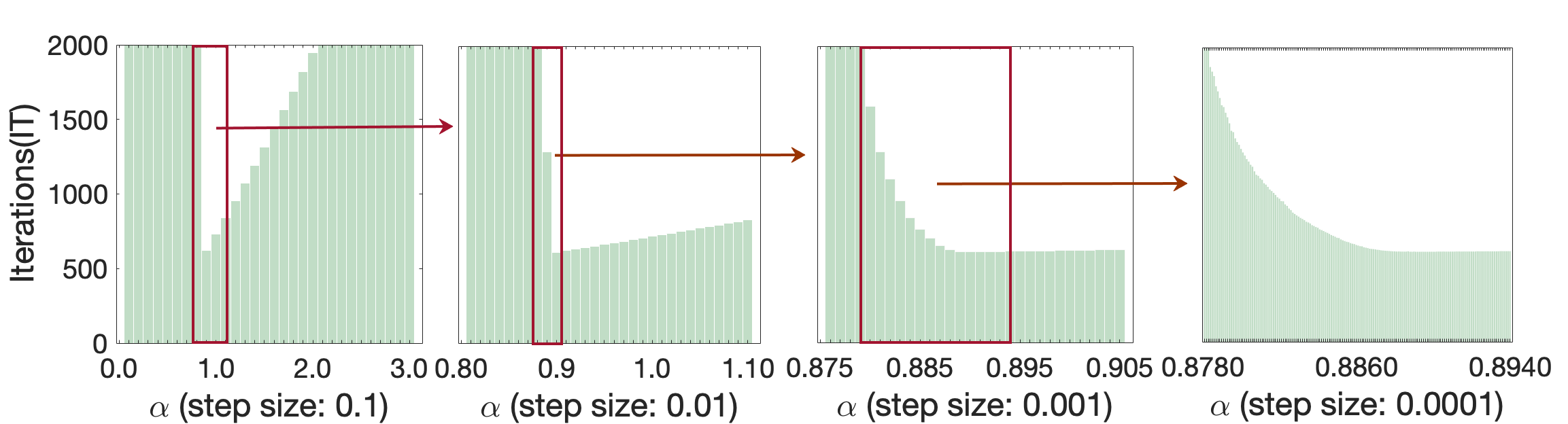}
	\caption{When $ n=64^3$, the variation of the iterations (IT) of the IHSS method
		on the traversal parameter $\alpha$ of different step size. \label{fig:e1_alpha}}
\end{figure}
Meanwhile, it should be emphasized that the efficiency of ADI schemes is very sensitive to the accuracy of splitting parameter $\alpha$.
As \Cref{fig:e1_alpha} shows, when $n=64^3$, the IT still changes as the traversing step size becomes $10^{-4}$.
However, it will cost much more traversing CPU time if one uses a smaller step size to find a more optimal $\alpha$.
Therefore, obtaining an efficient performance of the IHSS scheme may require undergoing very expensive computation by finding a relatively optimal splitting parameter $\alpha$.

Compared with the IHSS method, the practical GADI-HS scheme is much more efficient in combination with the GPR approach.
The GPR method can predict a relatively optimal $\alpha$ through training a mapping from $n$ to $\alpha$.
The small set of training data is obtained by the offline computation for small-scale systems. The learned mapping can directly predict a relatively optimal splitting parameter $\alpha$ for large systems without any extra computational amount, as \Cref{fig:e1_GRP} shows.
Therefore in the practical (online) computation, the practical GADI-HS scheme can obtain the solution with an efficient one-shot computation and does not consume traversal CPU time anymore. And with accurate $\alpha$, the practical GADI-HS method can efficiently solve large linear systems, even of more than ten million orders. As an example, \Cref{fig:convergentCurves3D} plots the convergent curve of IHSS and practical GADI-HS when $n^3=64^3$.
\begin{figure}[htbp]
	\centering
	\footnotesize{
		\includegraphics[width=5cm]{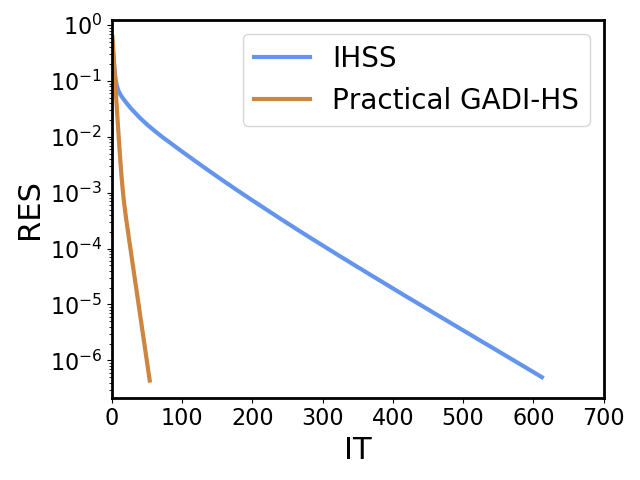}
		\caption{
			The convergent curves of IHSS and practical GADI-HS when  $n^3=64^3$ for solving 3D convection-diffusion equation.} \label{fig:convergentCurves3D}
	}
\end{figure}

Meanwhile, from Tables \ref{tab:PGADIHSsetting} and \ref{tab:e1_e2_result}, one can find that the maximum order in the training set is $66^3$ (287496), while the order of matrix in the test set can reach the level of ten million ($216^3$). It means that the scale of the predictable $n$ can attain about $35$ times. It means that the order of matrix in the given training set is $n$, and the prediction ability of the GPR method can reach a scale of $35n$. These results demonstrate that the GPR model has a high generalization capability.

We further plot the CPU cost of the HSS, IHSS, GADI-HS, and
practical GADI-HS methods as $n^3$ increases from $8^3$ to $216^3$ as shown in \Cref{fig:e2_results}.
The concrete CPU times can be found in Tables \ref{tab:alpha_q} and \ref{tab:e1_e2_result}.
For the IHSS method, \Cref{fig:e2_results} only presents its best performance in
solving \cref{eq:3DConvDiff} and ignores the CPU time of finding $\alpha^*$.
From these results, one can find that the practical GADI-HS method costs less
CPU time. For example, when $n^3=128^3$,
the IHSS method (594.28 seconds) spends about five times CPU time more than the practical
GADI-HS scheme (113.82 seconds) does. It should be emphasized that the IHSS method
costs much traversal time of over 10000 seconds to obtain a good performance
(traversing $\alpha^*$ in [0.8,0.9] with a step size of 0.01), while the practical GADI-HS method does not have traversing time due to the GPR method. If one considers all costs,
the practical GADI-HS saves over 100 times the computational cost for the case of $n^3=128^3$.
%
%The largest calculated
%matrix scale of the Practical GADI-HS method has reached the million level, and
%the speed is the fastest. For example, in \Cref{fig:e2_results}, when
%$n=128^3(2097152)$, the CPU time under the optimal parameter consumed by the IHSS
%method takes about 6 times than Practical GADI-HS method, and the larger of the scale
%of a matrix, the more significant acceleration ratio. Additionally, the traversal
%time to obtain the curve of the IHSS method immensely exceeds that of the Practical
%GADI-HS method because each time IHSS method requires traverse parameters of
%different scale matrices to get the optimal parameter, while the GRP method for
%selecting the parameters of Practical GADI-HS, only the traversal time of the
%training set is consumed, which is a once and for all method of selecting
%parameters.
\begin{figure}[htbp]
	\centering
	\includegraphics[width=8cm]{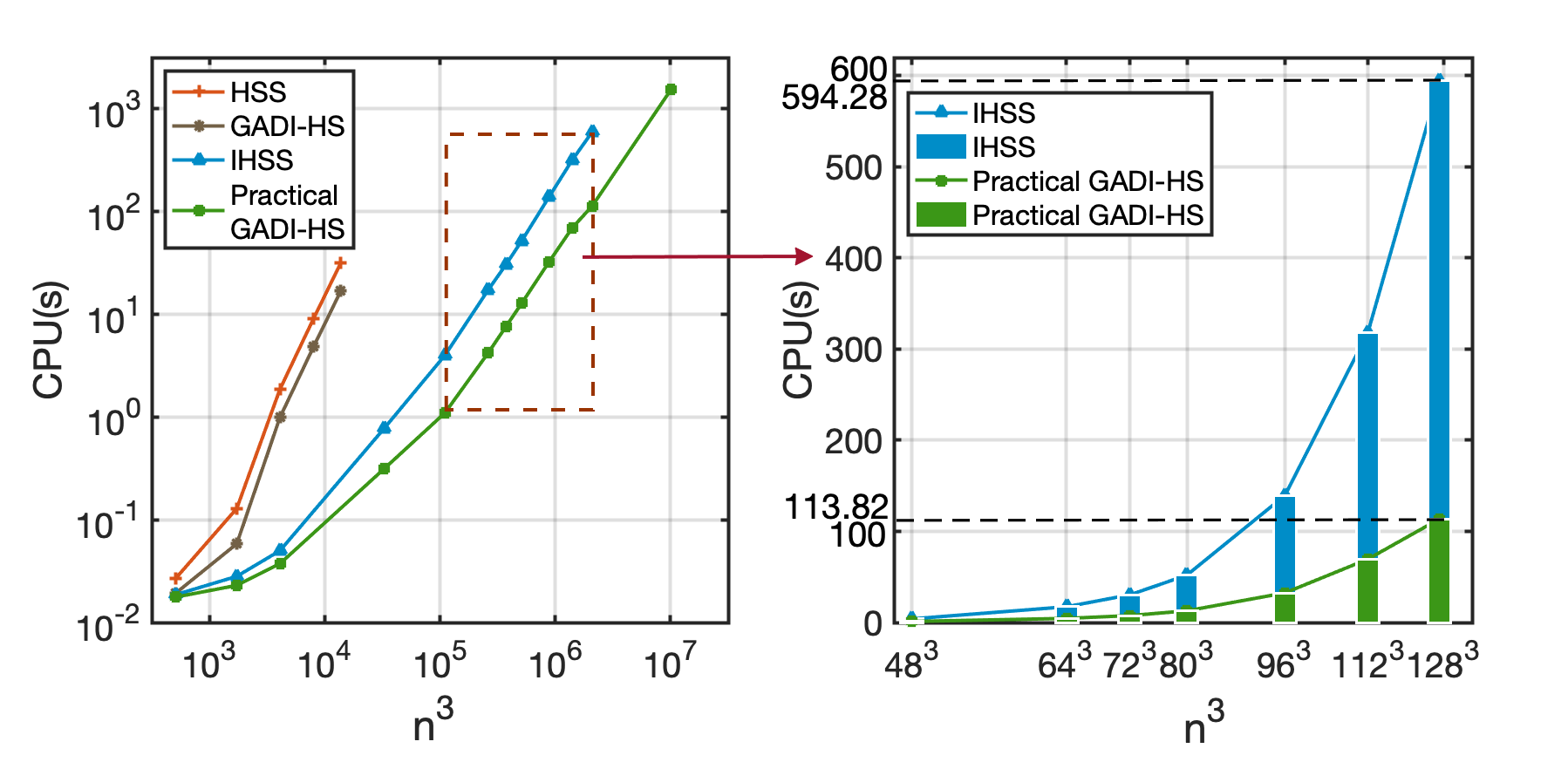}
	\caption{Left: The CPU time of HSS, IHSS, GADI-HS, and practical GADI-HS
	methods for different $n$. Right: The comparison of CPU time between IHSS and practical GADI-HS methods for $n$ from 48 to 128.\label{fig:e2_results}}
\end{figure}

Moreover, we compare the practical GADI-HS method with the GMRES, preconditioned GMRES, and iLU methods for solving the 3D convection-diffusion equation, as \Cref{tab:IHSS_PGADI-HS_GMRES_iLU} shows. It can be seen that the practical GADI-HS method is more efficient than the GMRES and iLU methods, and has similar performance with some efficient preconditioned GMRES, such as the Toeplitz preconditioner proposed by Strang \cite{strang1986proposal}. It should be stressed that the GADI framework is the first work to greatly increase the performance of splitting methods in solving large sparse linear systems. It has enormous potential to further promote the performance
of the GADI methods corresponding the structure of linear systems.

\begin{table}[htbp]
	\centering
	\footnotesize{
	\caption{
	A comparison of the practical GADI-HS, GMRES, preconditioned GMRES, and iLU methods for solving 3D
	convection-diffusion equation.\label{tab:IHSS_PGADI-HS_GMRES_iLU}}
		\begin{tabular}{|c|c|c|c|c|c|c|c|}
			\hline
			&
			\multicolumn{2}{c|}{Practical GADI-HS} & \multicolumn{2}{c|}{GMRES} & \multicolumn{2}{c|}{Preconditioned GMRES} & iLU
			\\
			\hline
			{$n^3$} & {IT} & {CPU(s)} & {IT} & {CPU(s)} & {IT} & {CPU(s)} & {CPU(s)} \\
			\hline
			{$32^3$ (32768)}&  23 & 0.26 & 97 & 0.46 & 17 & 0.30 & 32.56 \\
			{$48^3$ (110592}&  33 & 1.09 & 141 & 1.91 & 24 & 1.22 & 1166  \\
			{$64^3$ (262144)}& 54 & 4.21 & 184 & 7.44 & 35 & 4.26 &$>$5000\\
			{$96^3$ (884736)}& 110 & 32.48 & 265 & 75.25 & 57 & 26.13 & \\
			\hline
		\end{tabular}
		}
\end{table}

\subsubsection{Predicting the optimal splitting parameter of the IHSS method by the GPR}
\label{subsubsec:IHSSGRP}

 Here we apply the GPR method to predict the optimal splitting parameter of the IHSS scheme. The training, test, and retrained data sets are given in
\Cref{tab:GPR_IHSS}. The implementation of the GPR method is the same as the
\Cref{subsubsec:PGADIHS}. \Cref{fig:GPR_IHSS} gives the inference curve
of the GPR.
\begin{table}[!hbtp]
	\centering
	\footnotesize{
	\caption{
	The training, test, and retrained sets of the GPR approcah for
	predicting the optimal splitting parameter of the IHSS method
	in solving 3D convection-diffusion equation.}\label{tab:GPR_IHSS}
		\begin{tabular}{|c|c|}
			\hline
			& $n: 2 \sim 10$, $\Delta n=2$ \hspace{0.3cm}\\
		    {Training set} & $n: 12 \sim 20$, $\Delta n=4$\\
		    			   & $n: 24 \sim 48$, $\Delta n=8$\\
						   & $n: 64 $ \hspace{1.8cm} \\
			\hline
			{Test set} & $n: 1\sim 120$, $\Delta n =1$     \\			
			\hline
			{Retrained set}&$n: 70 \sim 120$, $\Delta n=6$ \\
			\hline
		\end{tabular}
	}
\end{table}
\begin{figure}[hbtp]
	\centering
	\includegraphics[width=5cm]{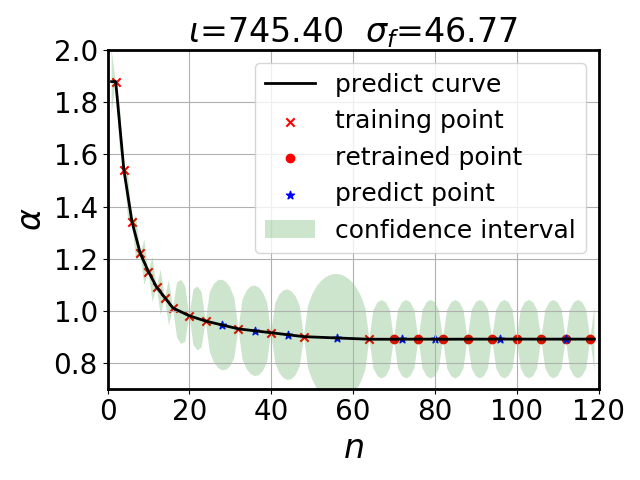}
	\caption{The $\alpha$ regression curve of the IHSS method.
\label{fig:GPR_IHSS}}
\end{figure}
\begin{table}[htbp]
	\centering
	\footnotesize{
	\caption{
	A performance comparison of traversing parameter $\alpha$ and predicting
	parameter $\alpha^*$ by the GPR in the IHSS method when
	solving 3D convection-diffusion equation.\label{tab:e1_IHSS_result}}
	\begin{tabular}{|c|c|c|c|c|}
			\hline
			{$n^3$} & {$\alpha$} & {$\alpha^*$} & {IT} & {IT$^*$}  \\
			\hline
			{$28^3$ (21952)}& 0.94 & 0.9450 & 149 & 148\\
			{$36^3$ (46656)}& 0.92& 0.9225& 226 & 226\\
			{$44^3$ (85184)}& 0.91 &0.9075& 317& 319 \\
			{$56^3$ (175616)}& 0.90 & 0.8950 & 484 & 488 \\
			{$72^3$ (373248)}& 0.89 &0.8899 & 759 & 759 \\
			{$80^3$ (512000)}& 0.89 & 0.8900 & 920 & 920 \\
			{$96^3$ (884736)}& 0.89 & 0.8896 & 1284 & 1285\\
			{$112^3$ (1404928)}& 0.88 & 0.8884 & 1702 & 1682 \\
			\hline
		\end{tabular}
	}
\end{table}

%\begin{figure}[tbhp]
%\centering
%\subfloat[$\epsilon_{\max}=5$]{\label{fig:a}\includegraphics[width=6cm]{./IHSSGRP/IHSS_e1_af.png}}
%\subfloat[$\epsilon_{\max}=0.5$]{\label{fig:b}\includegraphics[width=6cm]{./IHSSGRP/IHSS_e1_af.png}}
%\caption{Example figure using external image files.}
%\label{fig:testfig}
%\end{figure}

\Cref{tab:e1_IHSS_result} shows the prediction results and the performance of
IHSS method. In \Cref{tab:e1_IHSS_result}, the second column is the optimal
traversing parameter $\alpha$ with step size $0.01$, the third column is the
predicted parameter $\alpha^*$ obtained by the GPR, and the fourth (IT) and fifth
(IT$^*$) columns are iteration steps corresponding to splitting
parameters $\alpha$ and $\alpha^*$, respectively.
From \Cref{tab:e1_IHSS_result}, we can find that the GPR method can predict accurate splitting
parameters which are consistent with the traversing method. As a result, the
iteration steps IT and IT$^*$ of the IHSS method are nearly the same for convergence.
More significantly, when the dimension $n^3$ of the linear system becomes large to
$112^3$, the GPR method can predict more accurate splitting parameter which can
promote the performance of the IHSS. These results demonstrate that the GPR can be
applied to other ADI schemes.

\begin{remark}
Here, let us give a remark to show the reasonableness of using the dimension $n$ as the input variable in the GRP method.
\Cref{tab:condition_number} shows the relationship between the condition number and dimensional coefficient matrix $A$ of the 3D convection-diffusion equation. It can be seen that the condition number of coefficient matrix $A$ increases as the dimension $n$ increases. This means that minimal or maximal eigenvalues of $A$ are depend on $n$. The goal of GPR is to learn a mapping from a readily available quantity to the optimal splitting parameter in the GADI framework. The dimension $n$ is the required quantity and can reflect some essential features of the coefficient matrix, such as eigenvalues. Thus, it is reasonable to choose $n$ as the input variable.
\end{remark}

\begin{table}[htbp]
  \footnotesize{
  \caption{
  The condition number of $A$ with different $n$ when
	solving 3D convection-diffusion equation.\label{tab:condition_number}}}
  \centering
  \footnotesize{
    \begin{tabular}{|c|c|c|c|c|c|}
      \hline
      $n^3$ & $8^3$ & $16^3$ & $32^3$ & $48^3$ & $64^3$ \\
      \hline
      Condition number & 53.0767 & 192.6712 & 727.6907 & 1.6079e+03 & 2.8294e+03 \\
      \hline
    \end{tabular}
    }
\end{table}

\subsection{The continuous Sylvester equation}

In this subsection, we consider the continuous Sylvester equation \cref{eq:cSe}. The sparse matrices $A$ and $B$ have the following structure:
%$
%A=B=M+2rN+{100I}/{(n+1)^{2}},
%$
\begin{equation*}
A=B=M+2rN+\frac{100}{(n+1)^{2}}I,
\end{equation*}
where $r$ is a parameter which controls Hermitian dominated or skew-Hermitian dominated of the matrix, $M$, $N \in \mathbb{C}^{n\times n}$ are tridiagonal
matrices $M=\mbox{Tridiag}(-1, 2,- 1)$, and $N=\mbox{Tridiag}(0.5, 0, -0.5)$.
We apply the HSS and the GADI-AB methods to solve \cref{eq:cSe} for $r=0.01$, $0.1$, and $1$, respectively.
All iterative methods are started from zero matrix and stopped once
the current residual norm satisfies
$
{\|R^{(k)}\|_{F}}/{\|R^{(0)}\|_{F}} \leq 10^{-6},
$
where $R^{(k)}=C-AX^{(k)}-X^{(k)}B$.

\begin{table}[htbp]
	\caption{Results of the HSS and GADI-AB methods for solving the continuous Sylvester equation. The relatively optimal splitting parameter $\alpha^*$
	$(\omega^*)$ is obtained by traversing the interval $(0,3]$ $([0,2))$ with a step size of $0.01$ $(0.1)$.	
	\label{tab:results_e1}}
	\centering
	\footnotesize{
		\begin{tabular}{|l|ccc|ccc|ccc|}
			\hline
			&\multicolumn{3}{c|}{} &\multicolumn{3}{c|}{{HSS}}&\multicolumn{3}{c|}{} \\
			&\multicolumn{3}{c|}{{$r=0.01$}} &
			\multicolumn{3}{c|}{{$r=0.1$}}&\multicolumn{3}{c|}{{$r=1$}}  \\
		   \hline
			{$n$}&	{IT}  & {CPU}(s) &{$\alpha^*$} &
			{IT}  & {CPU}(s) &{$\alpha^*$} &{IT}  &
			{CPU}(s) &{$\alpha^*$}  \\
			\hline
			{16}&23&0.0059&1.23&22&0.0051&1.28&14&0.0032&1.57\\
			{32}&43&0.0307&0.64&41&0.0281&0.64&21&0.0146&1.07\\
			{64}&83&0.3512&0.33&69&0.3053&0.32&31&0.1242&0.87\\
			{128}&160&3.7057&0.17&112&2.6082&0.19&46&1.0632&0.74\\
			{256}&312&42.9613&0.09&163&24.0084&0.14&68&10.3410&0.58\\
			\hline
			&\multicolumn{3}{c|}{}
			&\multicolumn{3}{c|}{{GADI-AB}}&\multicolumn{3}{c|}{} \\
			&\multicolumn{3}{c|}{{$r=0.01$}} &
			\multicolumn{3}{c|}{{$r=0.1$}}&\multicolumn{3}{c|}{{$r=1$}}  \\
			 \hline
			{$n$}&	{IT}  & {CPU}(s) &({$\alpha^*,\omega^*$}) &
			{IT}  & {CPU}(s)
			&({$\alpha^*,\omega^*$})&{IT}  & {CPU}(s) &({$\alpha^*,\omega^*$}) \\
			\hline
			{16}&12&0.0001&(1.18,0.0)&12&0.0001&(1.18,0.0)&8&0.0001&(1.87,0.0)\\
			{32}&22&0.0007&(0.62,0.0)&21&0.0007&(0.65,0.0)&12&0.0004&(1.28,0.1)\\
			{64}&42&0.0038&(0.33,0.0)&38&0.0035&(0.36,0.0)&16&0.0011&(0.97,0.1)\\
			{128}&81&0.0443&(0.17,0.0)&63&0.0370&(0.22,0.0)&21&0.0063&(0.76,0.1)\\
			{256}&157&0.4499&(0.09,0.0)&90&0.2457&(0.15,0.0)&29&0.0605&(0.54,0.1)\\
			\hline
		\end{tabular}
	}
\end{table}
\Cref{tab:results_e1} lists the simulation results of the HSS and GADI-AB methods.
For the continuous Sylvester equation, the HSS and GADI-AB methods do not have theories
to estimate relatively optimal splitting parameters.
Here we can use the traversal approach to estimate the relatively optimal
splitting parameters.
The optimal parameter $\alpha^*$ is obtained by traversing
the interval $(0,3]$ with a step size of $0.01$, while
the optimal parameter $\omega^*$ is obtained through traversing the
interval $[0,2)$ with a step size of $0.1$.
From \Cref{tab:results_e1}, one can find that the GADI-AB method is much more
efficient than the HSS method in terms of the IT and CPU.
For example, when $n=256$,
\Cref{fig:e2_ar} shows the acceleration ratio of the GADI-AB method over the HSS
scheme in terms of CPU time.
\begin{figure}[htbp]
	\centering
	\includegraphics[width=4.5cm]{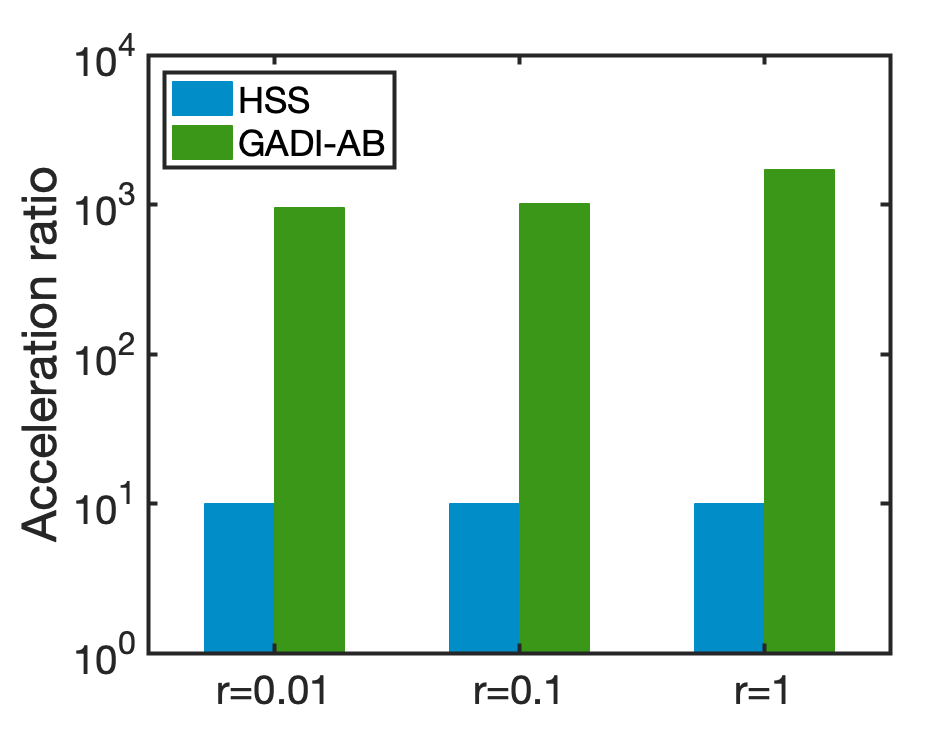}
	\caption{
		The acceleration ratio of the GADI-AB against the HSS when $n=256$.
		\label{fig:e2_ar}
	}
\end{figure}
To clearly demonstrate the acceleration
ratio, we take the CPU cost of the HSS as a reference (set as ``10'').
As shown in \Cref{fig:e2_ar}, the GADI-HS method spends much less CPU time than the
HSS scheme does for different $r$ with acceleration ratio $95.49$ ($r=0.01$),
$97.71$ ($r=0.1$), and $170.92$ ($r=1$).
\Cref{fig:convergentCurvesSylvester} gives the corresponding convergent curves of HSS and GADI-AB.
\begin{figure}[hbtp]
	\centering
	\includegraphics[width=4cm]{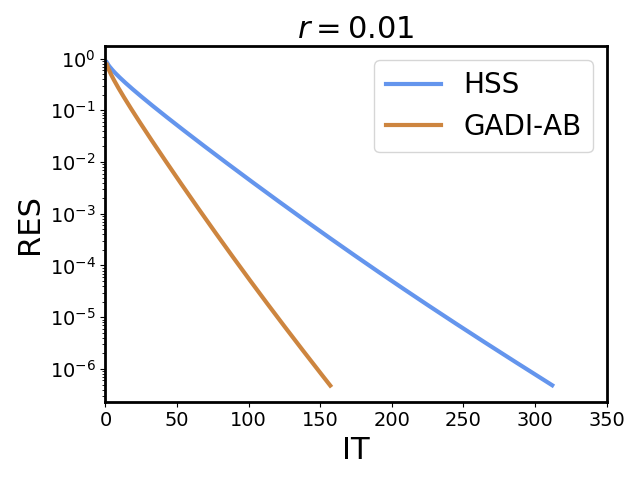}
	\includegraphics[width=4cm]{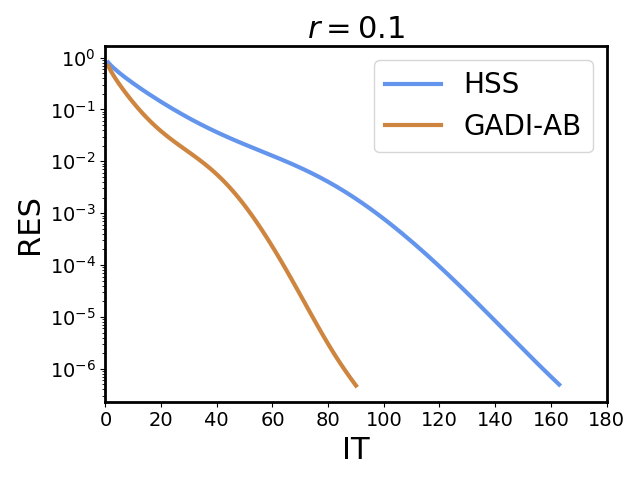}
	\includegraphics[width=4cm]{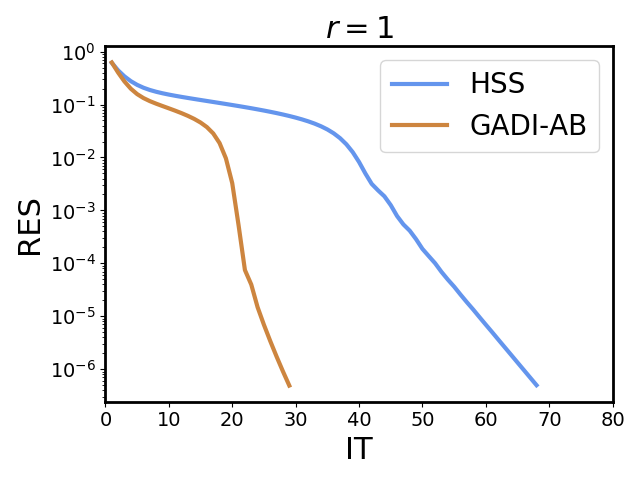}
	\caption{The convergent curves of HSS and GADI-AB for
			solving the continuous Sylvester equation \cref{eq:cSe} when $n=256$.
			\label{fig:convergentCurvesSylvester}}
\end{figure}

These numerical results demonstrate that the HSS and the GADI-AB methods can
efficiently solve the continuous Sylvester equation \cref{eq:cSe} with relatively optimal
splitting parameters. And our proposed GADI-AB method indeed saves much
more time compared with the HSS method.
However, obtaining relatively optimal splitting parameters through the traversal
method is expensive.
For instance, when $n=128, r=0.01$, the HSS method takes over 6000 seconds to
traverse $\alpha$. The traversal time will greatly increase as $n$ increases.
In the following, we will apply the GPR method to accelerate the GADI-AB
method and calculate larger-scale matrix equations.

Analogously, we find that $\omega$ is an insensitive parameter. Therefore we obtain a relatively optimal $\omega_p$ through the traversing method for small-scale problems. And we use the GPR method to predict the relatively optimal $\alpha$.
\Cref{tab:GADI-AB_GRPset} gives the test, training, and retrained data sets.
\Cref{fig:e1_GADI-AB_GR} shows the parameter regression process of the GADI-AB
method, and \Cref{tab:e1_GADI-AB_result} presents the prediction results obtained through curves.
$\alpha_p$ denotes the prediction parameter obtained by the GPR method.

\begin{table}[!hbtp]
	\centering
	\caption{The training, test, and retrained sets used in the GPR algorithm for the GADI-AB method.
		\label{tab:GADI-AB_GRPset}}
	\footnotesize{
		\begin{tabular}{|c|c|c|c|}
			\hline
			&{$r = 0.01$}&    {$r = 0.1$} &  	{$r = 1$}   \\	
			\hline
			{Training} & $n: 4\sim 20$, $\Delta n=4$ & $n: 4\sim 20$, $\Delta n=4$& $n:
			4\sim 20$, $\Delta n=4$\\	{set}  	& $n: 24\sim 72$, $\Delta n=8$ & $n:
			24\sim 72$, $\Delta n=8$ & $n: 24\sim 72$, $\Delta n=8$ \\
			& $n: 80\sim 112$, $\Delta n=16$& $n: 80\sim 112$, $\Delta n=16$& $n:
			80\sim 112$,
			$\Delta n=16$ \\
			\hline
			{Test set} & $n: 1\sim 500$, $\Delta n=1$  & $n: 1\sim 500$, $\Delta n=1$& $n:
			1\sim 1000$, $\Delta n=1$\\			
			\hline
			{Retrained}&$n: 120\sim 500$, & $n: 120\sim 500$, & $n: 120\sim 1000$, \\
			{ set}&$\Delta n=30$& $\Delta n= 10$ & $ \Delta n=50$ \\
			\hline
		\end{tabular}
	}
\end{table}
\begin{figure}[hbtp]
	\centering
	\includegraphics[width=4.25cm]{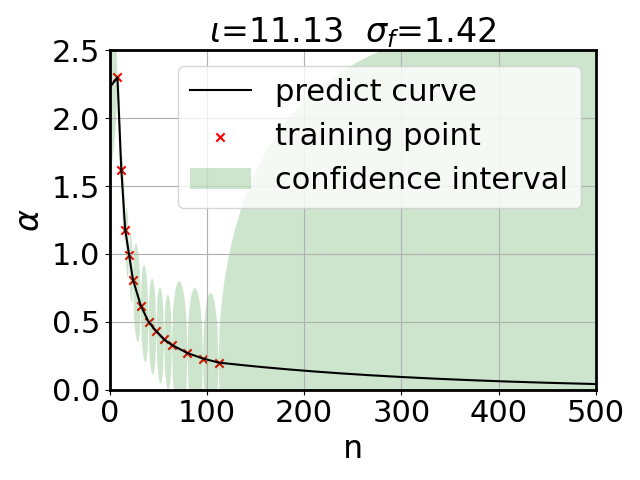}
	\includegraphics[width=4.25cm]{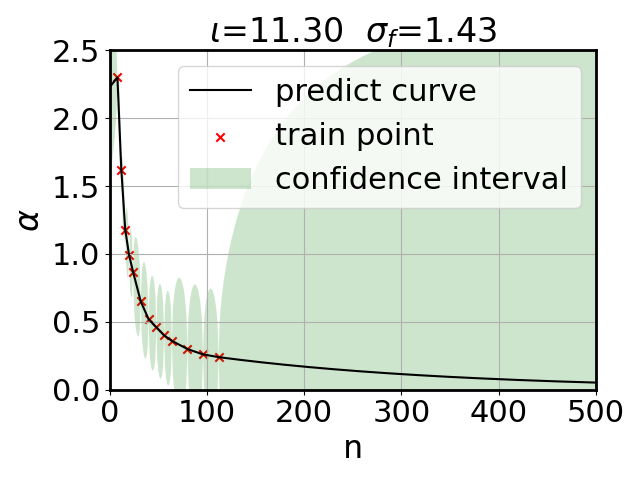}
	\includegraphics[width=4.25cm]{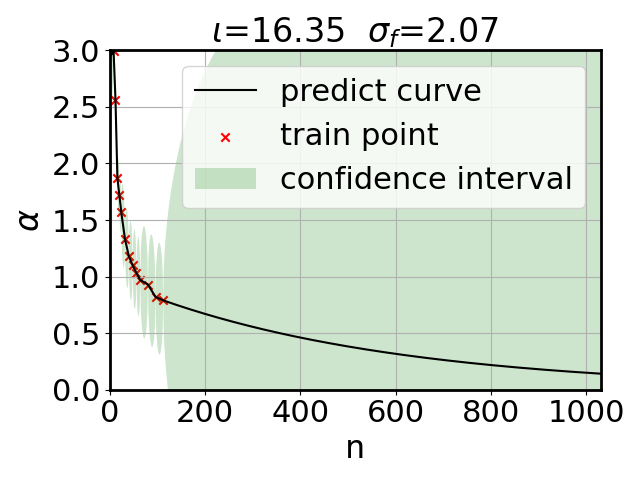}
	\includegraphics[width=4.25cm]{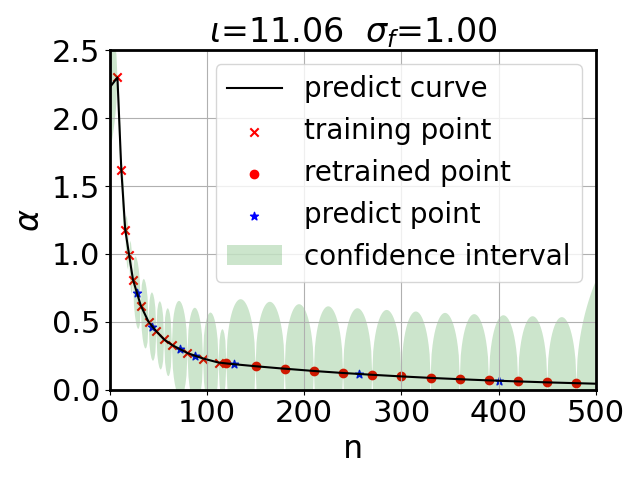}
	\includegraphics[width=4.25cm]{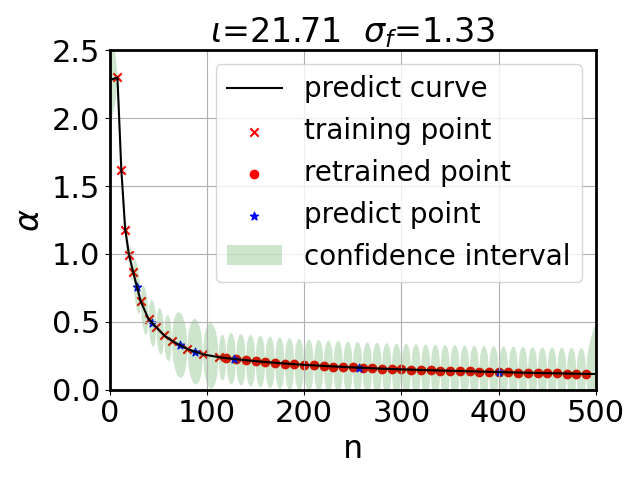}
	\includegraphics[width=4.25cm]{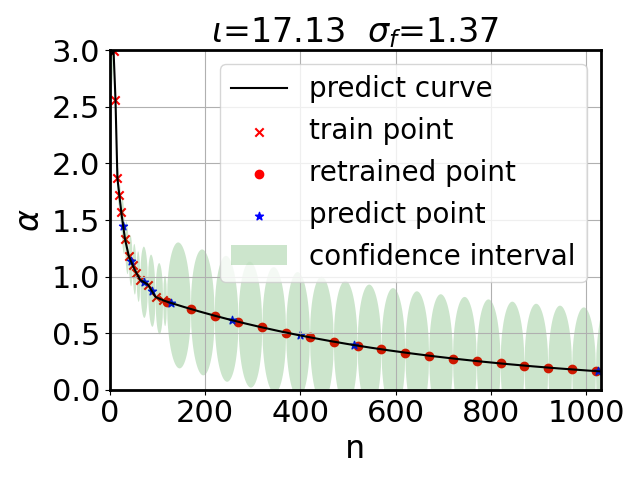}
	\caption{The regression curve of $\alpha$ for the GADI-AB method (Left:
		$r=0.01$; Middle: $r=0.1$; Right: $r=1$).
		\label{fig:e1_GADI-AB_GR}}
\end{figure}
\begin{table}[!htbp]
	\centering
	\caption{Results of the GADI-AB method solving the continues Sylvester equation by predicting parameter $\alpha_p$ of GPR.\label{tab:e1_GADI-AB_result}}
	\footnotesize{
		\begin{tabular}{|c|c|c|c|c|}
			\hline
			&{$n$} & $( \alpha_p,
			\omega_p)$&{$IT_{\alpha_p}$}&{CPU}(s)\\
			\hline		
			&{256}&(0.1161,0.0)&200& 0.4827\\
			${r=0.01}$&{400}&(0.0651,0.0)&258& 1.4994\\
			&{512}&(0.0421,0.0)&319& 4.6850\\
			\hline
			&{256} &(0.1621,0.0)&93& 0.2348\\
			&{400} &(0.1299,0.0)&110& 0.5985\\
			${r=0.1}$&{512} &(0.1144,0.0)&120& 1.5475\\
			&{1024}&(0.0668,0.0)&201&18.8272\\
			\hline
			&{256} &(0.6170,0.1)&30& 0.0605\\
			&{400} &(0.4808,0.1)&36& 0.1706\\
			${r=1}$&{512} &(0.3961,0.1)&40& 0.3263\\
			&{1024} &(0.1654,0.1)&92& 6.6008\\
			&{2048}&(0.0285,0.1)&534&289.6384\\
			\hline
		\end{tabular}
	}
\end{table}

From \Cref{fig:e1_GADI-AB_GR}, we also find that putting the predicted data into a retrained set can shrink the confidence interval.
\Cref{fig:e1_GADI-AB_GR} and \Cref{tab:e1_GADI-AB_result} show that
the scale of the predictable $n$ can reach 4.57 times ($r=0.01$), 9.14 times ($r=0.1$), and 18.29 times ($r=1$).
These results demonstrate that the GPR method has a good generalization capability for solving the continuous Sylvester equation.
The GPR method also provides an accurate optimal parameter prediction without an expensive consumption.
The resulting method allows us to solve much larger Sylvester matrix equations (see \Cref{{tab:e1_GADI-AB_result}}) than these methods in \Cref{tab:results_e1} do through the traversal way to obtain relatively optimal splitting parameters.

To further demonstrate the performance of our proposed method, we also compare the existing ADI algorithms with the GADI-AB method for solving the
continuous Sylvester equation \cref{eq:cSe}.
\Cref{tab:compwithOthers} summarizes these results. The data of other
methods all come from corresponding references.
From \Cref{tab:compwithOthers}, one can find that the proposed GADI-AB method is the most efficient, tens to thousands of times faster than the existing methods even ignoring the consumption of obtaining a relatively optimal splitting parameters under almost the same hardware and software environments.
\begin{table}[htbp]
	\centering
	\caption{Results of ADI methods of solving continuous Sylvester equation with $n=256$.
		\label{tab:compwithOthers}}
	\footnotesize{
		\begin{tabular}{|l|l|c|c|c|r|}
			\hline
			&{Paper}
			&{Hardware Environment}&{Method}&{IT}&{CPU(s)}\\
			\hline
			&\cite{bai2011hermitian}&---&HSS&203&44.67\\
			&\cite{bai2011hermitian}&---&SOR&310&244.41\\
			&\cite{wang2013positive}&1.4 GHz, 2GB RAM&HSS&518&1736.07\\
			{$r=0.01$}&\cite{wang2013positive}&1.4 GHz, 2GB RAM&TSS&288&1005.85\\
			&\cite{ZHENG2014145}&2.2 GHz, 8GB RAM&NSS&118&20.12\\
			&\cite{ke2014a}&2.4 GHz, 2GB RAM&NSCG&287&101.05\\
			&\cite{ke2014a}&2.4 GHz, 2GB RAM&PNSCG&16&5.81\\
			\rowcolor{mgray}
			&\textbf{Our}&\textbf{2.3 GHz, 8GB RAM}&\textbf{GADI-AB}&\textbf{157}&\textbf{0.45}\\
			\hline
			&\cite{bai2011hermitian}&---&HSS&156&33.89\\
			&\cite{bai2011hermitian}&---&SOR&304&236.69\\
			&\cite{wang2013positive}&1.4 GHz, 2GB RAM&HSS&274&961.41\\
			{$r=0.1$}&\cite{wang2013positive}&1.4 GHz, 2GB RAM&TSS&227&813.45\\
			&\cite{ke2014a}&2.4 GHz, 2GB RAM&NSCG&170&61.22\\
			&\cite{ke2014a}&2.4 GHz, 2GB RAM&PNSCG&75&14.84\\
			\rowcolor{mgray}
			&\textbf{Our}&\textbf{2.3 GHz, 8GB RAM}&\textbf{GADI-AB}&\textbf{90}&\textbf{0.25}\\
			\hline
			&\cite{bai2011hermitian}&---&HSS&85&20.49\\
			&\cite{bai2011hermitian}&---&SOR&256&205.23\\
			{$r=1$}&\cite{wang2013positive}&1.4 GHz, 2GB RAM&HSS&95&654.86\\
			&\cite{wang2013positive}&1.4 GHz, 2GB RAM&TSS&84&269.89\\
			&\cite{ZHENG2014145}&2.2 GHz, 8GB RAM&NSS&69&15.12\\
			\rowcolor{mgray}
			&\textbf{Our}&\textbf{2.3 GHz, 8GB RAM}&\textbf{GADI-AB}&\textbf{29}&\textbf{0.06}\\
			\hline
		\end{tabular}
	}
\end{table}

\section{Conclusions}\label{sec:conclu}

In this paper, we propose a GADI framework to solve large-scale sparse linear systems.
The new proposed framework can unify most existing ADI methods,
and can derive new methods as shown in \Cref{tab:GADIframework}.
In this work, we present three new ADI methods, including the GADI-HS, the
practical GADI-HS, and the GADI-AB schemes.
To address the challenge of how to choose optimal splitting parameters of splitting methods, we present a data-driven method, the GPR method, to predict relatively optimal splitting parameters, which greatly improves the efficiency of ADI methods.
Combining with the GADI framework and the GPR method, we can
address large linear sparse systems within an efficient one-shot computation.
Numerical results demonstrate that the proposed methods are faster by tens to thousands
of times than the (inexact) HSS-type methods.
Moreover, our proposed methods can solve much larger linear systems
that these existing ADI methods have not reached.

There are still lots of works based on the proposed methods. For instance, one is to apply the GPR method to predicting optimal parameters for more splitting schemes.
The second interesting work is to extend the GPR method to predict optimal
multiparameter methods. The third one is to develop more
efficient ADI schemes from the proposed framework to solve much larger structured
linear systems. The fourth interesting point is to extend the proposed methods to
nonlinear systems.\\

{{\bf Acknowledgement.} We sincerely thank the editor and anonymous referees for the insightful comments and suggestions. Those
comments are all valuable and very helpful for revising and improving our paper. We appreciate Qi Zhou for the kind help in revising our paper. }

\section*{Appendix A. Proofs.}
\section*{Appendix A.1. Proof of \Cref{le:pra GADI-HS}}\label{app:lemma_split_equivalence}
\begin{proof}
	Set
	\begin{equation}\label{eq:pra GADI splitting}
	\left\{ \begin{aligned}
	M_1&=\alpha I+H,\\
	N_1&=\alpha I-S,
	\end{aligned}\right.\quad
	\left\{\begin{aligned}
	M_2&=(S-\alpha I)[S-(1-\omega)\alpha I]^{-1}(\alpha I+S),\\
	N_2&=(S-\alpha I)[S-(1-\omega)\alpha I]^{-1}(\alpha I+S)-A.
	\end{aligned}\right.
	\end{equation}
	By using the first splitting scheme in \cref{eq:pra GADI splitting}, we have
	$
	(\alpha I+H)\tilde x^{(k+\frac{1}{2})}=(\alpha I-S)\tilde x^{(k)}+b,$
	i.e.,
	\begin{equation}\label{eq:pra GADI splitting1}
	\tilde x^{(k)}=(\alpha I-S)^{-1}[(\alpha I+H)\tilde x^{(k+\frac{1}{2})}-b].
	\end{equation}
	Additionally, from the practical GADI framework \cref{eq:pra GADI-HS}, we have
	\begin{equation}\label{eq:pra GADI splitting2}
	[S-(1-\omega)\alpha I]^{-1}(\alpha
	I+S)\tilde x^{(k+1)}=\tilde x^{(k)}+[S-(1-\omega)\alpha I]^{-1}(2-\omega)\alpha
	\tilde x^{(k+\frac{1}{2})}.
	\end{equation}
	We substitute \cref{eq:pra GADI splitting1} into \cref{eq:pra GADI splitting2},
	and utilize the practical GADI framework \cref{eq:pra GADI-HS} to obtain
	$\tilde x^{(k+1)}$
	\begin{equation*}
	\begin{aligned}
	& (S-\alpha I)[S-(1-\omega)\alpha I]^{-1}(\alpha
	I+S)\tilde x^{(k+1)}\\
	=&(S-\alpha I)[S-(1-\omega)\alpha I]^{-1}(\alpha
	I+S)\tilde x^{(k+\frac{1}{2})}-A\tilde x^{(k+\frac{1}{2})}+b.
	\end{aligned}
	\end{equation*}
	Considering the second splitting in \cref{eq:pra GADI splitting}, it is evident that
	$$ M_2 \tilde x^{(k+1)}=N_2\tilde x^{(k+\frac{1}{2})}+b.$$
	This implies that the practical GADI framework (\Cref{alg:pra GADI-HS})
	is equivalent to the assumption conditions \cref{eq:pra GADI-HS splitting}.
	Using \cref{eq:pra GADI-HS splitting}, we have
	\begin{equation*}
	\begin{aligned}
	\tilde x^{(k+1)}=&~\tilde x^{(k+\frac{1}{2})}+M_2^{-1}(\tilde r^{(k+\frac{1}{2})}+\tilde q^{(k+\frac{1}{2})})=M_2^{-1}N_2\tilde x^{(k+\frac{1}{2})}+M_2^{-1}b+M_2^{-1}\tilde q^{(k+\frac{1}{2})}\\
	=&~M_2^{-1}N_2[\tilde x^{(k)}+M_1^{-1}(\tilde r^{(k)}+\tilde p^{(k)})]+M_2^{-1}b+M_2^{-1}\tilde q^{(k+\frac{1}{2})}\\
	=&~M_2^{-1}N_2M_1^{-1}N_1\tilde x^{k}+M_2^{-1}(N_2M_1^{-1}+I)b+M_2^{-1}(N_2M_1^{-1}\tilde p^{(k)}+\tilde q^{(k+\frac{1}{2})}),
	\label{eq:pra GADIx^k+1}
	\end{aligned}
	\end{equation*}
	i.e., \cref{eq:pra GADI-HS scheme}. Thus, we complete the proof.
\end{proof}

\section*{Appendix A.2. Proof of \Cref{le:bound of GADI-HS}}\label{app:lemma_bound}

\begin{proof}
	By \cref{eq:HSiterMatrix} and the similarity invariance of matrix spectrum, we obtain
	\begin{align*}
	& \rho (T(\alpha,\omega))=\rho ((\alpha I+S)^{-1}(\alpha I+H)^{-1}(\alpha
	^{2} I+HS-(1-\omega)\alpha A))
	\\
	& \leq \alpha\|(\alpha I+S)^{-1}(\alpha I+H)^{-1}(\alpha I-(1-\omega)
	A)\|_{2}+\|(\alpha I+S)^{-1}(\alpha I+H)^{-1}HS\|_{2}
	\\
	& \leq \alpha \|(\alpha I+S)^{-1}\|_{2}\|(\alpha I+H)^{-1}\|_{2}(\alpha+|1-\omega||\|A\|_2)
	+\|(\alpha I+H)^{-1}H\|_{2}\|S(\alpha I+S)^{-1}\|_{2}.
	\end{align*}
	Scaling the last two items appropriately yields
	\begin{equation*}
	\begin{aligned}
	&\alpha \|(\alpha I+S)^{-1}\|_{2}\|(\alpha
	I+H)^{-1}\|_{2}(\alpha+|1-\omega|\|A\|_2)\\
	=&~\alpha \max \limits_{\lambda _{i} \in \lambda (H)} \Big |\frac{1}{\alpha
		+\lambda_{i}}\Big |\max \limits _{\sigma _{i} \in \sigma
		(S)}\Big|\frac{1}{\sqrt {\alpha ^{2}+\sigma _{i}^{2}}}\Big
	|(\alpha+|1-\omega|\|A\|_2)\\
	\leq&~ \frac{
		\alpha}{\alpha+\lambda_{min}}\cdot\frac{1}{\alpha}(\alpha+|1-\omega|\|A\|_2)=\frac{\alpha+|1-\omega|\|A\|_2}{\alpha+\lambda_{min}}.
	\end{aligned}
	\end{equation*}
	Since $H$ is a Hermitian positive definite matrix, then $\lambda_{i}>0$  and
	\begin{equation*}
	\begin{aligned}
	&\|(\alpha I+H)^{-1}H\|_{2}\|S(\alpha I+S)^{-1}\|_{2}\\
	=&\max \limits_{\lambda _{i}\in \lambda (H)}\Big |
	\frac{\lambda_{i}}{\alpha +\lambda_{i}}\Big|\max \limits_{\sigma _{i}\in
		\sigma (S)}\Big | \frac{\sigma _{i}}{\sqrt {\alpha ^{2}+\sigma
			_{i}^{2}}}\Big |
	=\frac{\lambda_{max}}{\alpha+\lambda_{max}}\cdot\frac{\sigma_{max}}{\sqrt{\alpha^2+\sigma_{max}^2}}\leq \frac{\lambda_{max}\sigma_{max}}{(\alpha+\lambda_{min})\alpha}.
	\end{aligned}
	\end{equation*}
	Therefore, we obtain
	\begin{equation*}
	\begin{aligned}
	\rho {(T(\alpha,\omega))}\leq\frac{\alpha^2+\alpha|1-\omega|\|A\|_2+\lambda_{max}\sigma_{max}}{\alpha(\alpha+\lambda_{min})}.
	\end{aligned}
	\end{equation*}
\end{proof}

\section*{Appendix A.3. Proof of \Cref{th:GADI-HSalpha}}
\begin{proof}
	(i) From \Cref{th:|mu| and |lamba|}(i), when $|\lambda_k|^2\leq a$ and $0\leq \omega<2$, then
	$$\rho(M(\alpha))\leq \rho(T(\alpha,\omega))<1.$$ If $\alpha^*$ and $\omega^*$ are the quasi-optimal parameters in this case, we have
	\begin{equation*}
	\rho(	T(\alpha^*,\omega^*))=\rho(M(\alpha^*)).
	\end{equation*}
	Furthermore, due to $\rho(T(\alpha,0))=\rho(M(\alpha))$ and
	\cref{eq:|mu_k|}, it is obvious that
	$$\omega^*=\argmin_{\omega}\Big\{\frac{1}{2}[\omega+(2-\omega)\sigma(\alpha)]\Big\}=0.$$
	The GADI-HS method reduces to the HSS iterative method, from \Cref{lem:HSSalpha}; then
	\begin{equation*}
	\alpha^*=\argmin_{\alpha}\Big\{\frac{1}{2}[\omega+(2-\omega)\sigma(\alpha)]\Big\}=\argmin_{\alpha}\Big\{\sigma(\alpha)\Big\}=\sqrt{\lambda_{min}\lambda_{max}}.
	\end{equation*}
	
	(ii) From \Cref{th:|mu| and |lamba|}(ii), when $a<|\lambda_k|^2$
	and $0< \omega<\frac{4a^2-4a+4b^2}{(1-a)^2+b^2}$, we have
	$$\rho(T(\alpha,\omega))<\rho(M(\alpha))<1.$$ Thus, we consider the bound of
	$\rho(T(\alpha,\omega))$ in inequality \cref{eq:bound of T}
	\begin{equation*}
	\rho(T(\alpha,\omega))\leq \delta(\alpha,\omega)=
	\frac{\alpha^2+\alpha|1-\omega|\|A\|_2+\lambda_{max}\sigma_{max}}{\alpha(\alpha+\lambda_{min})}.
	\end{equation*}
	Our objective is to minimize $\delta(\alpha,\omega)$.
	
	When $\alpha$ is fixed, $\delta(\alpha,\omega)$ reaches the minimum at
	$w=1$
	\begin{equation*}
	\omega^*=\argmin_{\omega}\left\{\delta(\alpha,\omega)\right\}=1,
	\end{equation*}
	then
	\begin{equation}\label{eq:min_alpha}
	\alpha^*=\argmin_{\alpha}\left\{\frac{\alpha^2+\lambda_{max}\sigma_{max}}{\alpha(\alpha+\lambda_{min})}\right\}:=\argmin_{\alpha}\left\{\frac{\alpha^2+p}{\alpha(\alpha+\lambda_{min})}\right\},
	\end{equation}
	where $p=\lambda_{max}\sigma_{max}$. The first-order optimal condition of
	\cref{eq:min_alpha} means
	$$
	\frac{\partial
		\delta(\alpha,1)}{\partial\alpha}=\frac{2\alpha(\alpha^2+\alpha\lambda_{min})-(\alpha^2+p)(2\alpha+\lambda_{min})}{(\alpha+\lambda_{min})^2\alpha^2}=0.$$
	Solving the above equation, $\alpha^*= \argmin_{\alpha}\left\{\delta(\alpha,\omega)\right\}=
	\frac{p+\sqrt{p^2+\lambda_{min}^2p}}{\lambda_{min}}$. From \cref{eq:min_alpha},
	we have
	\begin{equation*}
	\begin{aligned}
	\delta(\alpha^*,\omega^*)
	=\frac{2p^2+2\lambda_{min}^2p+2p\sqrt{p^2+\lambda_{min}^2p}}{2p^2+2\lambda_{min}^2p+2p\sqrt{p^2+\lambda_{min}^2p}+\lambda_{min}^2\sqrt{p^2+\lambda_{min}^2p}}<1.
	\end{aligned}
	\end{equation*}
\end{proof}

\section*{Appendix A.4: Two-dimensional parabolic equation}

Two-dimensional parabolic equationConsider the following 2D parabolic equation
\begin{equation}
	-u_{x_1 x_1}-u_{x_2 x_2}+2u_{x_1 x_2}+u_{x_1}=f(x_1, x_2), \quad (x_1, x_2)\in \Omega,
	\label{eq:2DPara}
\end{equation}
on the unit cube $\Omega=[0,1]\times[0,1]$ with homogeneous Dirichlet boundary
condition. We use the centered difference method to
discretize the parabolic equation \cref{eq:2DPara}, and obtain the linear system $Ax=b$. The coefficient matrix is
\begin{equation}
	A=I\otimes T_1+D_1\otimes T_2+D_2\otimes T_3,
	\label{eq:2Dlinearsys}
\end{equation}
where $D_1, D_2, T_1, T_2, T_3$ are tridiagonal matrices defined by
$$D_1=\mbox{Tridiag} (0,0,1),~ D_2=\mbox{Tridiag} (1,0,0),~ T_1=\mbox{Tridiag}
(-1-\beta,4,-1+\beta),$$
$$T_2=\mbox{Tridiag} (-1/2,-1,1/2), ~ T_3=\mbox{Tridiag}
(1/2,-1,-1/2), ~ \beta=1/(2n+2),$$ with $n$ is the degree of freedom along each dimension.
$x\in\mathbb{R}^{n^2}$  is the unknown vector of discretizing $u(x_1,x_2)$.
	$b\in\mathbb{R}^{n^2}$  is the discretization vector of $f(x_1,x_2)$ which
is determined by choosing the exact solution $u(x_1, x_2)=\sin( \pi x_1) \sin (\pi x_2)$.
All tests are started with the zero vector.
All iterative methods are terminated if the relative residual error satisfies
$\mbox{RES}={\|r^{(k)}\|_{2}}/{\|r^{(0)}\|_{2}} \leq 10^{-6}$,
where $r^{(k)}=b-A x^{(k)}$ is the $k$-step residual.

\begin{table}[!htbp]
	\centering
	\footnotesize{
	\caption{ A comparison of the HSS and the GADI-HS methods for solving 2D parabolic equation with theoretical quasi-optimal splitting parameters.\label{tab:alpha_2DPDE} }
	\centering
		\begin{tabular}{|c|c|c|c|c|c|c|}
			\hline
			&\multicolumn{3}{c|}{HSS} &
			\multicolumn{3}{c|}{GADI-HS} \\
			\hline
			{$n^2$} &{$\alpha_q$}&	{IT}  & {CPU(s)} &{$(\alpha_q, \omega_q)$}  &
			{IT}  & {CPU(s)} \\
			\hline
			{$16^2$}& 0.6156 & 77 & 0.0054  & (0.1158, 1.0) & 37 & 0.0035\\
			{$32^2$}& 0.3050 & 140 & 0.1243 & (0.0603, 1.0) & 64 & 0.0580 \\
			{$64^2$}& 0.1501 & 257 & 4.8332 & (0.0307, 1.0) & 114 & 2.1430\\
			{$96^2$}& 0.0991 & 373 & 35.6514 & (0.0206, 1.0) & 163 & 15.8759\\
			\hline
		\end{tabular}
		}
\end{table}
Firstly, we compare the performance of the GADI-HS and the HSS in solving the above
linear algebra equation $Ax = b$.
The parameters in the HSS and the GADI-HS methods are obtained
from Lemma 3.4 and Theorem 3.5, respectively.
\Cref{tab:alpha_2DPDE} presents corresponding numerical results, where ``IT" and ``CPU"
denote the required iterations and the CPU time (in seconds), respectively.
From \Cref{tab:alpha_2DPDE}, one can find that the GADI-HS method spends less than half of
CPU times compared with the HSS method. It is consistent with the prediction by Theorem 3.5
that shows the convergence speed of the GADI-HS scheme is faster than the HSS method.
These results demonstrate that the GADI-HS scheme derived from the new GADI framework
accelerates the convergence speed in solving \cref{eq:2DPara}.

Next, we use the Practical GADI-HS and the IHSS methods to solve much larger linear systems.
Meanwhile, we employ the GPR method to predict the splitting parameters in the
Practical GADI-HS method.
The inner iteration is terminated if residuals satisfy
\begin{equation*}
\|p^{(k)}\|_{2} \leq 1 \times
10^{-\delta_{H}}\|r^{(k)}\|_{2},~~\|q^{(k)}\|_{2}\leq 1 \times
10^{-\delta_{S}}\|r^{(k+\frac{1}{2})}\|_{2},
\end{equation*}
where $\delta _{H}$ and $\delta _{S}$ are controllable tolerances in the inner iteration
to balance the Hermitian and skew-Hermitian parts in the linear sub-problems.
In these tests, we set $\delta_{H}=\delta_{S}=2$.

For the IHSS method, there has been no theoretical approach to estimate splitting
parameters. Alternatively, the splitting parameters can be obtained by traversing
method or by the GRP method. Section 4.1.3 has demonstrated that the GRP method can
predict accurate splitting parameter of IHSS method as the traversing method does.
Here, we still use the traversing method to obtain the relatively accurate splitting parameter
$\alpha^*$ in traversing interval $(0,3]$ with a step size of $0.01$.
And we use the GPR method to predict the splitting
parameter $\alpha$ for the Practical GADI-HS scheme.
\Cref{tab:PGADIHSsetting2D} gives the training, test, and
retrained data sets.
$\alpha$ in the training data set of the GPR approach is produced by
traversing parameters, but only for small-scale linear
systems, $n$ from $2$ to $100$ with different step size $\Delta n$, as shown in
\Cref{tab:PGADIHSsetting2D}.
\Cref{fig:GRP2D} shows the optimal parameter regression and prediction processes for
the Practical GADI-HS method. From the \Cref{fig:GRP2D}, we can find that
adding retrained data set in training can improve the prediction accuracy and strengthen the
the generalization ability of the regression model.
Meanwhile, we find that the parameter $\omega$ in the Practical GADI-HS
method is insensitive to the scale of the matrix. Therefore, we obtain the optimal
$\omega_p$ by the traversing method for small-scale problems.
\begin{table}[!hbtp]
	\centering
	\footnotesize{
	\caption{
	The training, test and retrained sets of the Practical GADI-HS method in GPR
	algorithm.\label{tab:PGADIHSsetting2D}}
		\begin{tabular}{|c|c|}
			\hline
							 & $n: 6 \sim 10$, $\Delta n=2$\\
			{Training set} & $n: 12 \sim 24$, $\Delta n=4$\\
						   & $n: 32 \sim 72$, $\Delta n=8$\\
						   & $n: 80 \sim 100$, $\Delta n=10$\\
			\hline
			{Test set} & $n: 1\sim 200$, $\Delta n =1$     \\			
			\hline
			{Retrained set}&$n: 110 \sim 200$, $\Delta n=10$ \\
			\hline
		\end{tabular}
	}
\end{table}
\begin{figure}[hbtp]
	\centering
	\includegraphics[width=4.5cm]{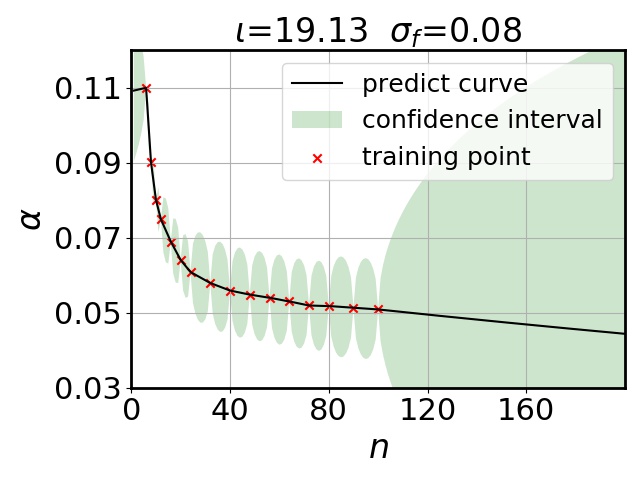}
	\includegraphics[width=4.5cm]{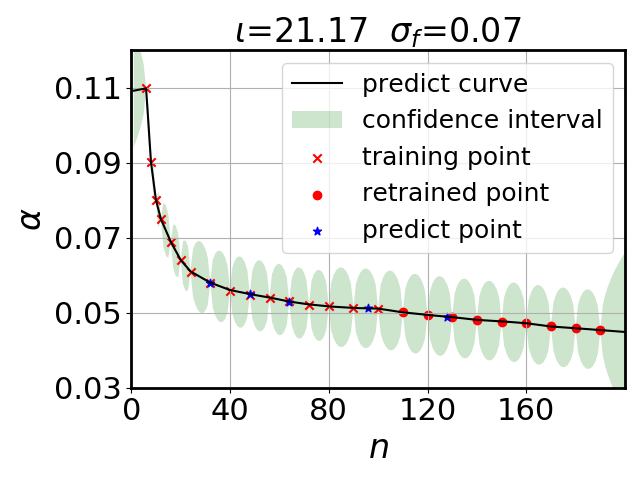}
	\caption{The regression curve of $\alpha$ against $n$ for the Practical GADI-HS method.
	\label{fig:GRP2D}}
\end{figure}

\Cref{tab:e1_e2_result2} shows the numerical results obtained by the IHSS and the
Practical GADI-HS methods for different scale discretization systems.
$IT_{CG}$ and $IT_{CGNE}$ denote the iterations for solving sub-systems in
inexact methods. These results demonstrate that the IHSS method can solve relatively large
linear systems with less CPU times than the HSS algorithm does when the relatively
optimal parameter $\alpha^*$ is used.
However, obtaining $\alpha^*$ in the IHSS costs lots of traversal times, as
shown in the last column in \Cref{tab:e1_e2_result2}.
For example, when $n=64^3$,  the traversal time of the IHSS method is about
$15000$ seconds.
\begin{table}[!hbpt]
	\centering
	\footnotesize{
	\caption{
	Results of the IHSS method ($\alpha^*$ obtained by traversing interval
	$(0,3]$ with a step size of $0.01$) and the Practical GADI-HS method (relatively
	optimal parameters $(\alpha_p,\omega_p)$ obtained by the GPR algorithm) for
	solving 2D parabolic equation.  \label{tab:e1_e2_result2}}
		\begin{tabular}{|c|lccc|c|}
			\hline
			{Method}&{$n^2$}&	{$\alpha^*$}&{IT}&{CPU}&{traversal} \\
			&&&(${IT_{CG}}$, ${IT_{CGNE}}$)&{(s)}&{CPU(s)} \\
			\hline
			&{$32^2(1024)$}& 0.92 & 403 (3.00, 1.00)& 0.08 & 2835.17\\
			{IHSS}&{$48^2(2304)$} & 0.91 & 853 (3.00, 1.00) & 0.36 & 6011.50\\
			&{$64^2(4096)$}& 0.91 & 1460 (3.00, 1.00) & 1.76 & 14881.66\\
			\hline
			&{$n^2$}&({$\alpha_p,\omega_p$})&{IT}&{CPU}&{traversal} \\
			&&&(${IT_{CG}, IT_{CGNE}}$)&{(s)}&{CPU(s)} \\
			\hline
			&{$32^2(1024)$}&(0.0578,1.9)& 34 (20.00, 2.00)& 0.05 &0\\
			&{$48^2(2304)$}&(0.0549,1.9)& 60 (23.00, 1.00)& 0.12 &0\\
			{Practical}&{$64^2(4096)$}&(0.0528,1.9)& 98 (23.00, 1.00)& 0.35 &0\\
			{GADI-HS}&{$96^2(9216)$}&(0.0512,1.9)& 203 (20.00, 1.00)& 1.71 &0\\
			& {$128^2(16384)$} &(0.0492,1.9)&  350 (18.00, 1.00)& 4.02 &0\\
			& {$256^2(65536)$}&(0.0430,1.9)& 1325 (12.00, 1.00)& 70.53 &0\\
			\hline
		\end{tabular}
	}
\end{table}
Compared with the IHSS method, the Practical GADI-HS scheme is much more efficient in
combination with the GPR approach. As an example, \Cref{fig:convergentCurves2D} gives
the convergent curve when $n^2=64^2$.
Therefore in the practical (online) computation,
the Practical GADI-HS scheme can obtain the solution with an efficient one-shot
computation and does not consume traversal CPU time anymore. And with accurately
predicted $\alpha$, the Practical GADI-HS method can efficiently solve large linear
systems.
\begin{figure}[hbtp]
	\centering
	\includegraphics[width=5cm]{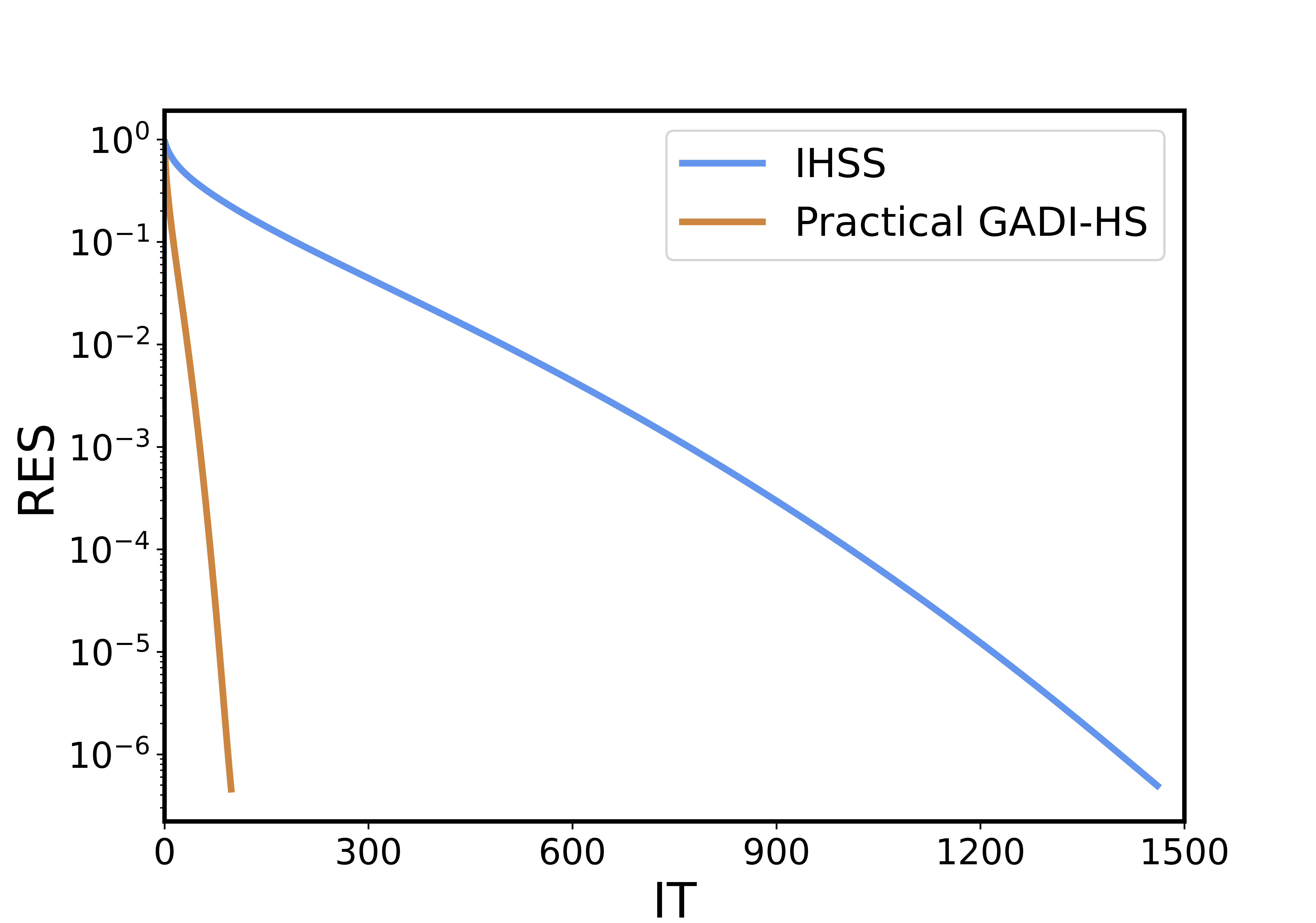}
	\caption{The convergent curves of the IHSS and the Practical GADI-HS in
	solving 2D parabolic equation when $n^2=64^2$.
	\label{fig:convergentCurves2D}}
\end{figure}

\section*{Acknowledgments}
We sincerely thank the editor and anonymous referees for the insightful comments and suggestions. Those comments are all valuable and very helpful for revising and improving our paper. We appreciate Qi Zhou for the kind help in revising our paper.

\end{document}